\documentclass[a4paper]{amsart}
\usepackage[utf8]{inputenc}
\usepackage[T1]{fontenc}
\usepackage{lmodern}
\usepackage{amssymb}
\usepackage[all]{xy}
\usepackage{nicefrac,mathtools,enumitem,amsmath}
\usepackage{microtype}

\usepackage{tikz}
\usetikzlibrary{matrix,arrows}
\tikzset{cd/.style=matrix of math nodes,row sep=2em,column sep=2em, text height=1.5ex, text depth=0.5ex}
\tikzset{cdar/.style=->,auto}

\tikzset{dar/.style={double,double equal sign distance,-implies}}%style for 2-arrows in triangles
\tikzset{mid/.style={anchor=mid}} % put labels on the arrow

\tikzset{triar/.style={anchor=mid,->}}%style for 1-arrows in triangles
\tikzset{tridar/.style={anchor=mid,double,double equal sign distance,-implies}}%style for 2-arrows in triangles
\tikzset{narrowfill/.style={inner sep=0pt, fill=white}}% style for nodes with filled background

\setlist[enumerate,1]{label=\textup{(\arabic*)}}
\setlist[enumerate,2]{label=\textup{(\alph*)}}

\usepackage[pdftitle={Colimits in the correspondence bicategory},
pdfauthor={Suliman Albandik and Ralf Meyer},
pdfsubject={Mathematics}
]{hyperref}
\usepackage[lite]{amsrefs}
\newcommand*{\MRref}[2]{ \href{http://www.ams.org/mathscinet-getitem?mr=#1}{MR \textbf{#1}}}

\newcommand*{\arxiv}[1]{\href{http://www.arxiv.org/abs/#1}{arXiv: #1}}

\renewcommand{\PrintDOI}[1]{\href{http://dx.doi.org/\detokenize{#1}}{doi: \detokenize{#1}}}

\BibSpec{book}{%
  +{}  {\PrintPrimary}                {transition}
  +{,} { \textit}                     {title}
  +{.} { }                            {part}
  +{:} { \textit}                     {subtitle}
  +{,} { \PrintEdition}               {edition}
  +{}  { \PrintEditorsB}              {editor}
  +{,} { \PrintTranslatorsC}          {translator}
  +{,} { \PrintContributions}         {contribution}
  +{,} { }                            {series}
  +{,} { \voltext}                    {volume}
  +{,} { }                            {publisher}
  +{,} { }                            {organization}
  +{,} { }                            {address}
  +{,} { \PrintDateB}                 {date}
  +{,} { }                            {status}
  +{}  { \parenthesize}               {language}
  +{}  { \PrintTranslation}           {translation}
  +{;} { \PrintReprint}               {reprint}
  +{.} { }                            {note}
  +{.} {}                             {transition}
 +{} { \PrintDOI}                   {doi}
  +{} { available at \url}            {eprint}
  +{}  {\SentenceSpace \PrintReviews} {review}
}

\theoremstyle{plain}
\newtheorem{theorem}[subsection]{Theorem}

\newtheorem{proposition}[subsection]{Proposition}

\newtheorem{corollary}[subsection]{Corollary}

\theoremstyle{definition}
\newtheorem{definition}[subsection]{Definition}

\theoremstyle{remark}

\newtheorem{example}[subsection]{Example}

\newcommand*{\braket}[2]{\langle#1, #2\rangle}% right inner products
\newcommand*{\congto}{\xrightarrow\sim}
\newcommand*{\Mult}{\mathcal M}%multiplier algebra
\newcommand*{\UM}{\mathcal U}% unitary multipliers

\newcommand{\C}{\mathbb{C}}
\newcommand{\N}{\mathbb{N}}
\newcommand{\Z}{\mathbb{Z}}

\newcommand{\T}{\mathbb{T}}

\newcommand*{\Cat}[1][C]{\mathcal #1}% category, usually C
\newcommand*{\Hilm}[1][E]{\mathcal #1}% Hilbert module
\newcommand*{\cm}{\mathcal{CM}}% crossed module
\newcommand*{\CP}{\mathcal{O}}% Cuntz-Pimsner algebra

\newcommand*{\Cont}{\mathrm C}%continuous functions
\newcommand*{\op}{\mathrm{op}}%opposite
\newcommand*{\prop}{\mathrm{prop}}%proper (correspondences)

\newcommand*{\Csttwocat}{\mathfrak C^*(2)}%bicategory of C*-algebras
\newcommand*{\Corrcat}{\mathfrak{Corr}}%bicategory of C*-algebras with correspondences

\newcommand*{\Star}{$^*$\nobreakdash-}
\newcommand*{\nb}{\nobreakdash}
\newcommand*{\Cst}{\mathrm C^*}%C*-algebra
\newcommand*{\defeq}{\mathrel{\vcentcolon=}}
\newcommand{\ket}[1]{{\lvert#1\rangle}}
\newcommand{\bra}[1]{{\langle#1\rvert}}

\newcommand{\Comp}{\mathbb K}
\newcommand{\Bound}{\mathbb B}

\newcommand{\Corr}{\mathfrak{Corr}}

\newcommand*{\norm}[1]{\lVert#1\rVert}

\DeclareMathOperator*{\colim}{colim}% colimit construction
\newcommand{\const}{\mathrm{const}}% constant diagram

\newcommand*{\Id}{\mathrm{id}}% identity map
\DeclareMathOperator{\Aut}{Aut}

\begin{document}
\title{Colimits in the correspondence bicategory}
\author{Suliman Albandik}
\email{albandik@uni-goettingen.de}
\author{Ralf Meyer}
\email{rmeyer2@uni-goettingen.de}

\address{Mathematisches Institut,
  Georg-August Universit\"at G\"ottingen, Bunsenstra{\ss}e 3--5,
  37073 G\"ottingen, Germany}


\begin{abstract}
  We interpret several constructions with \(\Cst\)\nb-algebras as
  colimits in the bicategory of correspondences.  This includes
  crossed products for actions of groups and crossed modules,
  Cuntz--Pimsner algebras of proper product systems, direct sums and
  inductive limits, and certain amalgamated free products.
\end{abstract}
\maketitle

\section{Introduction}
\label{sec:intro}

A basic idea of noncommutative geometry is to replace
ordinary quotient spaces by noncommutative generalisations.  For
instance, let a group~\(G\) act on a space~\(X\).  The orbit
space~\(X/G\) is often badly behaved as a topological space.  In
noncommutative geometry, it is replaced by the crossed product
\(\Cst\)\nb-algebra \(\Cont_0(X)\rtimes G\).  We may view the action
of~\(G\) on~\(X\) as a diagram of topological spaces.  The quotient
space is the colimit of this diagram.  We will exhibit the crossed
product for a group action as a colimit as well, in an appropriate
bicategory of \(\Cst\)\nb-algebras.  As this motivating example shows,
our bicategorical colimit construction leads to noncommutative
\(\Cst\)\nb-algebras even when we start with a diagram of locally
compact spaces.

The most concrete description of bicategories involves objects,
arrows, and \(2\)\nb-arrows, the composition of arrows and the
horizontal and vertical composition of \(2\)\nb-arrows.  We shall
emphasise a more conceptual definition: in a bicategory, \emph{sets}
of arrows between objects are replaced by \emph{categories} of arrows,
and the composition becomes a bifunctor.  Associativity and unitality
may hold exactly (strict \(2\)\nb-categories or just
\(2\)\nb-categories) or only up to natural equivalences of categories
that satisfy suitable coherence conditions (weak \(2\)\nb-categories
or bicategories, see \cites{Benabou:Bicategories,
  Leinster:Basic_Bicategories}).  We shall mostly work in the
bicategory~\(\Corrcat\) of \(\Cst\)\nb-algebra correspondences.  This
is introduced by Landsman in~\cite{Landsman:Bicategories} and studied
in some depth in~\cite{Buss-Meyer-Zhu:Higher_twisted}.

For simplicity, we also consider the bicategory~\(\Csttwocat\), which
is introduced in \cite{Buss-Meyer-Zhu:Higher_twisted}*{\S2.1.1}.  Its
objects are \(\Cst\)\nb-algebras, its arrows \(A\to B\) are
nondegenerate \Star{}homomorphisms \(A\to\Mult(B)\),
where~\(\Mult(B)\) denotes the multiplier algebra, and its
\(2\)\nb-arrows \(f_1\Rightarrow f_2\) for nondegenerate
\Star{}homomorphisms \(f_1,f_2\colon A\rightrightarrows \Mult(B)\) are
unitary multipliers \(u\in\UM(B)\) with \(uf_1(a)u^*=f_2(a)\) for all
\(a\in A\).  Since unitaries are invertible, the arrows \(A\to B\) and
\(2\)\nb-arrows between them in~\(\Csttwocat\) form a groupoid, not
just a category.

By the way, we may also restrict to nondegenerate \Star{}homomorphisms
\(A\to B\); this is like restricting to proper correspondences.  Since
a non-unital \(\Cst\)\nb-algebra contains no unitary elements, our
bicategory depends on using unitary \emph{multipliers}.  We need
nondegeneracy for our arrows so that they act on unitary multipliers.

What are diagrams in categories and their colimits?  Let \(\Cat\)
and~\(\Cat[D]\) be categories.  A \emph{diagram} in~\(\Cat[D]\) of
shape~\(\Cat\) is a functor \(\Cat\to\Cat[D]\).  Such diagrams are
again the objects of a category~\(\Cat[D]^{\Cat}\), with natural
transformations between functors as arrows.  Any object~\(x\)
of~\(\Cat[D]\) gives rise to a ``constant'' diagram \(\const_x\colon
\Cat\to\Cat[D]\) of shape~\(\Cat\).  The colimit \(\colim F\) of a
diagram \(F\colon \Cat\to\Cat[D]\) is an object of~\(\Cat[D]\) with
the following universal property: there is a natural bijection between
arrows \(\colim F\to x\) in~\(\Cat[D]\) and natural transformations
\(F\Rightarrow \const_x\) for all objects~\(x\) of~\(\Cat[D]\).  In
brief,
\begin{equation}
  \label{eq:colim_definition}
  \Cat[D](\colim F,x) \cong \Cat[D]^{\Cat}(F,\const_x).
\end{equation}

Now let \(\Cat\) and~\(\Cat[D]\) be bicategories.  As before, a
\emph{diagram} in~\(\Cat[D]\) of shape~\(\Cat\) is a functor (or morphism)
\(\Cat\to\Cat[D]\), as defined in, say, \cites{Benabou:Bicategories,
  Leinster:Basic_Bicategories}.  The functors \(\Cat\to\Cat[D]\) are
the objects of a bicategory~\(\Cat[D]^{\Cat}\); its arrows and
\(2\)\nb-arrows are the \emph{transformations} between functors and
the \emph{modifications} between transformations, see
\cites{Benabou:Bicategories, Leinster:Basic_Bicategories}.  These
definitions are repeated in our main
reference~\cite{Buss-Meyer-Zhu:Higher_twisted} in Definition 4.1 and
in \S4.2 and \S4.3.

Thus \(\Cat[D]^{\Cat}(F_1,F_2)\) for two diagrams \(F_1\) and~\(F_2\)
is now a \emph{category}, not just a set, with transformations
\(F_1\Rightarrow F_2\) as objects and modifications between them as
arrows.  Similarly, for two objects \(x_1\) and~\(x_2\)
of~\(\Cat[D]\), there is a category \(\Cat[D](x_1,x_2)\) of arrows
\(x_1\to x_2\) and \(2\)\nb-arrows between them.  Once again, there is
a constant diagram \(\const_x\) of shape~\(\Cat\) for any object~\(x\)
of~\(\Cat[D]\).  The bicategorical colimit is defined by the same
condition~\eqref{eq:colim_definition}, now interpreting~\(\cong\) as a
natural equivalence of categories.  An object \(\colim F\)
of~\(\Cat[D]\) with this property is unique up to equivalence if it
exists.

What do these definitions mean if \(\Cat=G\) is a group
and~\(\Cat[D]\) is the bicategory \(\Csttwocat\) described above?
First, diagrams in \(\Csttwocat\) are the twisted group actions in the
sense of Busby and Smith; this is observed
in~\cite{Buss-Meyer-Zhu:Higher_twisted}.  Transformations between such
diagrams are also described there.  In particular, a transformation
\(F\Rightarrow \const_D\) is a \emph{covariant representation} of the
twisted \(G\)\nb-action corresponding to~\(F\) in the multiplier
algebra of~\(D\).  A modification is a unitary intertwiner between two
covariant representations.  Hence the colimit and the crossed product
for the twisted action are characterised by the same universal
property, forcing them to be isomorphic.  As a result, if we replace
the category of spaces and maps by the bicategory \(\Csttwocat\), we
are led to enlarge the class of group actions to twisted actions, and
the crossed product construction appears as the natural analogue of a
``quotient'' in our bicategory.

Here we interpret many interesting constructions with
\(\Cst\)\nb-algebras as colimits.  Thus our new point of view unifies
several known constructions with \(\Cst\)\nb-algebras.  Most proofs
are as trivial as above: we merely make the universal property that
defines the bicategorical colimit explicit in a particular case and
recognise the result as the definition of a familiar
\(\Cst\)\nb-algebraic construction.

Instead of~\(\Csttwocat\), we mainly work in the correspondence
bicategory~\(\Corrcat\), which
is defined in \cite{Buss-Meyer-Zhu:Higher_twisted}*{\S2.2}.  Let \(A\)
and~\(B\) be \(\Cst\)\nb-algebras.  A \emph{correspondence} from~\(A\)
to~\(B\) is a Hilbert \(B\)\nb-module~\(\Hilm\) with a nondegenerate
\Star{}homomorphism from~\(A\) to the \(\Cst\)\nb-algebra of
adjointable operators on~\(\Hilm\).  An isomorphism between two such
correspondences is a unitary operator intertwining the left
\(A\)\nb-actions.  We let \(\Corrcat(A,B)\) be the groupoid of
correspondences from~\(A\) to~\(B\) and their isomorphisms.  The
composition is given by the bifunctors
\begin{equation}
  \label{eq:compose_corr}
  \Corrcat(B,C)\times\Corrcat(A,B)\to\Corrcat(A,C),\qquad
  (\Hilm,\Hilm[F])\mapsto \Hilm[F]\otimes_B \Hilm.
\end{equation}
This is associative and monoidal up to canonical isomorphisms, which
are part of the bicategory structure
(see~\cite{Buss-Meyer-Zhu:Higher_twisted}).  A
correspondence~\(\Hilm\) from~\(A\) to~\(B\) is \emph{proper} if the
left \(A\)\nb-module structure is through a map \(A\to\Comp(\Hilm)\)
to the \(\Cst\)\nb-algebra of compact operators.  Thus proper
correspondences with isomorphisms between them form a
subbicategory~\(\Corrcat_\prop\) of~\(\Corrcat\).  Our main results
will only hold for diagrams of proper correspondences, that is,
functors to~\(\Corrcat_\prop\).

Groups are categories with only one object.  At the other extreme are
discrete categories.  These are categories where all arrows are
identities, that is, sets viewed as categories.  Colimits in this case
are also called coproducts.  Whereas coproducts need not exist
in~\(\Csttwocat\), they are given by the \(\Cont_0\)\nb-direct sum in
the correspondence bicategory~\(\Corrcat\); this statement is a
standard additivity result about representations of
\(\Cont_0\)\nb-direct sums on Hilbert modules.  The nonexistence of
coproducts
in~\(\Csttwocat\) is one reason to prefer the correspondence
bicategory~\(\Corrcat\).  Moreover, since \(\Csttwocat\) is a
subbicategory of~\(\Corrcat\), we get more diagrams in~\(\Corrcat\)
than in~\(\Csttwocat\).

A functor \(G\to\Corrcat\) for a group~\(G\) is equivalent to a
saturated Fell bundle over~\(G\) (see
\cite{Buss-Meyer-Zhu:Higher_twisted}).  The colimit for such a functor
is the full \(\Cst\)\nb-algebra of sections of the corresponding Fell
bundle.

Crossed modules are a \(2\)\nb-categorical generalisation of
groups.  Their actions on \(\Cst\)\nb-algebras by automorphisms or
correspondences have been introduced in
\cites{Buss-Meyer-Zhu:Higher_twisted, Buss-Meyer:Crossed_products}.
Once again, the universal property of the colimit is the same as that
for the appropriate analogue of the crossed product in this context.

What happens for non-reversible dynamical systems?  Let~\(P\) be a
monoid, that is, a category with a single object.  A functor \(P\to
\Corrcat\) is the same as an essential product system over the
opposite monoid~\(P^\op\).  The change of direction comes
from~\eqref{eq:compose_corr}, where we tensor in reverse order to
conform to the usual conventions of composing maps.  Colimits for
product systems are remarkable because the universal property we get
is not always but often equivalent to a standard one.  More precisely,
if the product system is \emph{proper}, that is, all left actions in
the product system are through compact operators, then the colimit of
the corresponding diagram exists and is isomorphic to the
Cuntz--Pimsner algebra of the product system.  We get the ``absolute''
Cuntz--Pimsner algebra, not the popular modification by Katsura, and
we get there directly and never see the Cuntz--Toeplitz algebra along
the way.  This result on Cuntz--Pimsner algebras is the main idea
of~\cite{Albandik-Meyer:Product}.  We had originally
planned~\cite{Albandik-Meyer:Product} as an applications section
inside this article.  We were, however, convinced by
\(\Cst\)\nb-algebra colleagues to write down those results separately,
to make them accessible without category theory background.

Readers familiar with free products of \(\Cst\)\nb-algebras may have
been surprised that the
bicategory \(\Csttwocat\) is not closed under coproducts: already in
the usual category of \(\Cst\)\nb-algebras with \Star{}homomorphisms,
there is a coproduct, namely, the \emph{free} product.  This does not
cooperate with unitary multipliers, however, and fails to satisfy the
universal property for a coproduct in \(\Csttwocat\) or \(\Corrcat\).
This situation clears up when we consider pushouts.  Given two
\emph{nondegenerate} \Star{}homomorphisms \(B_1\leftarrow A\rightarrow
B_2\), their colimit in \(\Corrcat\) or \(\Csttwocat\) is the
amalgamated free product \(B_1 \star_A B_2\).  Free products without
amalgamation occur in the highly degenerate case \(A=0\).

Even more fundamental than pushouts are coequalisers.  These are
colimits of diagrams of the shape \(\Hilm_1,\Hilm_2\colon
A\rightrightarrows B\).  For instance, if \(A=B=\C\) and
\(\Hilm_i=\C^{n_i}\) for \(i=1,2\), then the coequaliser is the
universal \(\Cst\)\nb-algebra generated by elements~\(u_{jk}\) for
\(1\le j\le n_1\), \(1\le k\le n_2\), subject to the relations
\[
\sum_j u_{ij} u_{kj}^*=\delta_{i,k},\qquad
\sum_i u_{ij}^* u_{ik}=\delta_{j,k}
\]
for all \(1\le i,k\le n_1\) or all \(1\le j,k\le n_2\), respectively.
If \(n_1=n_2\), then this is the
\(\Cst\)\nb-algebra~\(U_n^\mathrm{nc}\) introduced by Brown and
studied further by McClanahan \cites{Brown:Ext_free,
  McClanahan:unitary_matrix, McClanahan:KK_twisted}.  This example
shows that coequalisers, even of very small diagrams, need not be
particularly well-behaved \(\Cst\)\nb-algebras.

Another situation we treat are inductive limits: the inductive limit
of a chain of \Star{}homomorphisms is also a colimit in~\(\Corrcat\),
even if some of these \Star{}homomorphisms are degenerate.

We also prove one general result here: any diagram of proper
correspondences, indexed by any bicategory, has a colimit.  We
describe this colimit by generators and relations, with the known
construction of Cuntz--Pimsner algebras of product systems as a model
case.  This model case also shows that something may go wrong for
diagrams involving non-proper correspondences.

\section{Colimits in bicategories}
\label{sec:colimits}

Let \(\Cat\) and~\(\Cat[D]\) be bicategories.  An object
of~\(\Cat[D]^{\Cat}\) is a functor (or morphism) \(\Cat\to\Cat[D]\);
it consists of
several objects, arrows and \(2\)\nb-arrows in~\(\Cat[D]\).  In the
\emph{constant diagram}, \(\const_x\colon \Cat\to\Cat[D]\), all these
objects are the same object~\(x\) of~\(\Cat[D]\), all the arrows are
the identity on~\(x\), and all \(2\)\nb-arrows are the identity
\(2\)\nb-arrow on~\(\Id_x\).

For instance, functors \(G\to\Csttwocat\) for a group~\(G\) are
identified with Busby--Smith twisted actions of~\(G\) on
\(\Cst\)\nb-algebras in
\cite{Buss-Meyer-Zhu:Higher_twisted}*{\S3.1.1}.  The constant diagram
\(\const_A\) for a \(\Cst\)\nb-algebra~\(A\) is the trivial
\(G\)\nb-action on~\(A\), with trivial twists.  Functors \(G
\to\Corrcat\) are identified with saturated Fell bundles in
\cite{Buss-Meyer-Zhu:Higher_twisted}*{\S3.1.1}.  A constant diagram
\(\const_A\) in~\(\Corrcat\) corresponds to the constant Fell bundle
with all fibres equal to~\(A\) and the constant multiplication and
involution.

\begin{definition}
  \label{def:colimit}
  Let \(\Cat\) and~\(\Cat[D]\) be bicategories and let \(F\colon
  \Cat\to\Cat[D]\) be a functor.  A \emph{cone} over~\(F\) is an
  object~\(x\) of~\(\Cat[D]\) with a transformation
  \(\vartheta_x\colon F\to \const_x\); a \emph{colimit} of~\(F\) is a
  universal cone over~\(F\), that is, an object~\(x\) of~\(\Cat[D]\)
  with a transformation \(\vartheta_x\colon F\to \const_x\), such that
  composition with~\(\vartheta_x\) induces equivalences of categories
  \[
  \Cat[D](x,y) \xrightarrow{\cong} \Cat[D]^{\Cat}(F,\const_y)
  \qquad\text{for all objects~}y\text{ of~}\Cat[D].
  \]
\end{definition}

If we are given natural equivalences \(\Cat[D](x,y) \cong
\Cat[D]^{\Cat}(F,\const_y)\), then the identity map in
\(\Cat[D](x,x)\) gives a transformation \(\vartheta_x\colon F\to
\const_x\), which is determined uniquely up to isomorphism; naturality
forces the equivalences \(\Cat[D](x,y)\to\Cat[D]^{\Cat}(F,\const_y)\)
to be composition with~\(\vartheta_x\).  Hence a colimit may also be
defined as an object~\(x\) of~\(\Cat[D]\) with natural equivalences of
categories \(\Cat[D](x,y)\cong \Cat[D]^{\Cat}(F,\const_y)\).

\begin{proposition}
  \label{pro:colim_functorial}
  The colimit is functorial: a transformation \(\Phi\colon F_1\to
  F_2\) induces an arrow \(\colim\Phi\colon \colim F_1\to\colim F_2\),
  and a modification \(\Phi_1\to\Phi_2\) induces a
  \(2\)\nb-arrow \(\colim\Phi_1\to \colim\Phi_2\), and these
  constructions are compatible with the composition bifunctor for
  transformations.
\end{proposition}

\begin{proof}
  Let \((x_1,\vartheta_1)\) and \((x_2,\vartheta_2)\) be colimits of
  \(F_1\) and~\(F_2\), respectively.  Transformations may be
  composed, so \(\vartheta_2\circ\Phi\) is an object of
  \(\Cat[D]^{\Cat}(F,\const_{x_2})\).  By the definition of the
  colimit, there is an arrow \(\colim \Phi\colon x_1\to x_2\) with
  \(\vartheta_2\circ\Phi \cong (\colim \Phi)\circ\vartheta_1\), and this
  arrow is unique up to equivalence.  Similarly, a modification
  \(\Phi_1\to\Phi_2\) induces a modification
  \(\vartheta_2\circ\Phi_1\to \vartheta_2\circ\Phi_2\),
  which gives a \(2\)\nb-arrow \(\colim \Phi_1\to \colim
  \Phi_2\).  Thus we get a functor \(\Cat[D]^{\Cat}(F_1,F_2)\to
  \Cat[D](\colim F_1,\colim F_2)\).  It is routine to check that
  this functor, up to equivalence, does not depend on choices and
  that the construction is compatible with the composition
  bifunctors in \(\Cat[D]^{\Cat}\) and~\(\Cat[D]\).
\end{proof}

\begin{corollary}
  \label{cor:colim_unique}
  Any two colimits of the same diagram are canonically equivalent.\qed
\end{corollary}

Equivalences in~\(\Csttwocat\) are \Star{}isomorphisms, those
in~\(\Corrcat\) are imprimitivity bimodules.  Hence colimits
in~\(\Csttwocat\) are unique up to isomorphism if they exist, whereas
colimits in~\(\Corrcat\) are only unique up to Morita--Rieffel
equivalence.

\section{Coproducts and products}
\label{sec:coproduct_product}

Coproducts are colimits of diagrams indexed by a category with only
identity morphisms.  Such a diagram is simply a map from some index
set~\(I\) to the objects of the category.  The following proposition
shows that the usual \(\Cont_0\)\nb-direct sum of \(\Cst\)\nb-algebras
is both a coproduct and a product of the set of objects \((A_i)_{i\in
  I}\) in \(\Corrcat\).  (We do not consider limits in this article
because it seems rare that they exist in~\(\Corrcat\).  We only
mention the result on products because its proof and statement are so
similar to the description of coproducts.)

\begin{proposition}
  \label{pro:product_coproduct_in_Corr}
  Let \(A_i\) for \(i\in I\) and~\(B\) be \(\Cst\)\nb-algebras.  Then
  \begin{align*}
    \Corrcat\left(\bigoplus_{i\in I} A_i,B\right) &\cong
    \prod_{i\in I} \Corrcat(A_i,B),\\
    \Corrcat\left(B,\bigoplus_{i\in I} A_i\right) &\cong
    \prod_{i\in I} \Corrcat(B,A_i).
  \end{align*}
\end{proposition}

\begin{proof}
  Given correspondences \(\Hilm_i\colon A_i\to B\), we may form the
  Hilbert \(B\)\nb-module \(\bigoplus_{i\in I} \Hilm_i\) and equip it
  with a nondegenerate left action of \(\bigoplus_{i\in I} A_i\) to
  get a correspondence from \(\bigoplus_{i\in I} A_i\) to~\(B\).
  Isomorphisms of correspondences \(\Hilm_i\to \Hilm_i'\) may be put
  together to an isomorphism of correspondences \(\bigoplus_{i\in I}
  \Hilm_i\to \bigoplus_{i\in I} \Hilm_i'\).  Thus we get a functor
  \begin{equation}
    \label{eq:coproduct_functor}
    \prod_{i\in I} \Corrcat(A_i,B) \to
    \Corrcat\left(\bigoplus_{i\in I} A_i,B\right).
  \end{equation}

  To show that~\eqref{eq:coproduct_functor} is an equivalence,
  consider a correspondence~\(\Hilm\) from \(\bigoplus_{i\in I}
  A_i\) to~\(B\).  Since the left action is nondegenerate, it
  extends to an action of the multiplier algebra
  of~\(\bigoplus_{i\in I} A_i\).  The latter is \(\prod_{i\in I}
  \Mult(A_i)\).  (The product is taken in the category of
  \(\Cst\)\nb-algebras, so it contains only bounded families.)  In
  particular, \(\Mult\bigl(\bigoplus_{i\in I} A_i\bigr)\) contains
  an orthogonal projection~\(p_i\) onto the \(i\)th summand for each
  \(i\in I\).  We have strict convergence \(\sum_{i\in I} p_i=1\).
  The projections~\(p_i\) act by orthogonal projections
  on~\(\Hilm\).  Let \(\Hilm_i\defeq p_i\Hilm\) be their images;
  these are Hilbert submodules on which~\(A_i\) acts
  nondegenerately, respectively.  Thus~\(\Hilm_i\) is a
  correspondence from~\(A_i\) to~\(B\).  Since \(\sum_{i\in I}
  p_i=1\), we have \(\bigoplus_{i\in I} \Hilm_i=\Hilm\).
  Thus~\(\Hilm\) belongs to the essential range of the
  functor~\eqref{eq:coproduct_functor}.  Furthermore, since any
  intertwining operator between two correspondences commutes with
  the left action of the multiplier algebra and hence with the
  projections~\(p_i\), it comes from a family of intertwining
  operators on the summands~\(\Hilm_i\); this shows that the
  functor~\eqref{eq:coproduct_functor} is fully faithful.
  Hence~\eqref{eq:coproduct_functor} is an equivalence of groupoids.
  This yields the first isomorphism, showing that \(\bigoplus_{i\in
    I} A_i\) is a coproduct of \((A_i)_{i\in I}\) in~\(\Corrcat\).

  Now consider a family of correspondences \(\Hilm_i\) from~\(B\)
  to~\(A_i\).  Let \(\bigoplus_{i\in I} \Hilm_i\) be the set of all
  families \((\xi_i)_{i\in I}\) with \(\xi_i\in\Hilm_i\) and
  \((i\mapsto \norm{\xi_i})\in\Cont_0(I)\).  This is a Hilbert module
  over \(\bigoplus_{i\in I} A_i\) by the pointwise operations.  The
  left actions of~\(B\) on the Hilbert modules~\(\Hilm_i\) give a
  nondegenerate left action of~\(B\) on \(\bigoplus_{i\in I}
  \Hilm_i\).  Thus we get a correspondence from~\(B\) to
  \(\bigoplus_{i\in I} A_i\).  This construction is natural
  with respect to isomorphisms of correspondences and hence gives a
  functor
  \begin{equation}
    \label{eq:product_functor}
    \prod \Corrcat(B,A_i) \to
    \Corrcat(B,\bigoplus A_i).
  \end{equation}
  Take a correspondence~\(\Hilm\) from~\(B\) to \(\bigoplus_{i\in I}
  A_i\).  For each \(i\in I\), \(\Hilm_i\defeq\Hilm\cdot
  A_i\subseteq\Hilm\) is a correspondence from~\(B\) to the
  ideal~\(A_i\) in~\(\bigoplus_{j\in I} A_j\).  Since these ideals
  are orthogonal, we have \(\Hilm\cong\bigoplus_{i\in I}\Hilm_i\).
  Thus~\(\Hilm\) belongs to the essential range
  of~\eqref{eq:product_functor}.  Since the decomposition
  \(\Hilm\cong\bigoplus_{i\in I}\Hilm_i\) is natural, the
  functor~\eqref{eq:product_functor} is fully faithful.
\end{proof}

Proposition~\ref{pro:product_coproduct_in_Corr} works because we may
take direct sums of correspondences to make things orthogonal.  In the
category of \(\Cst\)\nb-algebras with \Star{}homomorphisms as
morphisms, coproducts are free products, which are highly
noncommutative.  Since the coproduct in \(\Corrcat\) is unique up to
isomorphism in~\(\Corrcat\), that is, Morita--Rieffel equivalence, the
free product is \emph{not} a coproduct in~\(\Corrcat\) any more.  The
reason is that it is not compatible with isomorphisms of
correspondences: for a coproduct, we allow different unitaries
\(\Hilm_i\cong\Hilm'_i\) for all \(i\in I\).  Orthogonality of
the~\(\Hilm_i\) allows us to put two unrelated unitaries together.  In
the \(2\)\nb-category~\(\Csttwocat\), coproducts do not exist in
general for this reason: there are no orthogonal direct sums
in~\(\Csttwocat\), and free products do not behave well with respect
to \(2\)\nb-arrows.

\begin{example}
  \label{exa:coproduct_CC}
  We prove formally that the coproduct of two copies of~\(\C\)
  in~\(\Csttwocat\) does not exist.  Let~\(B\) be a
  \(\Cst\)\nb-algebra.  There is a unique arrow \(\C\to B\), namely,
  the unit map of~\(\Mult(B)\).  Thus there is a unique transformation
  from our coproduct diagram to~\(\const_B\), given by the unit map on
  both copies of~\(\C\).  A modification on this unique transformation
  is given by two unitaries \(u_1,u_2\in \Mult(B)\), one for each copy
  of~\(\C\), subject to no conditions.  If we also take \(B=\C\), then
  our groupoid of transformations is the two-torus group~\(\T^2\).

  Now assume that the \(\Cst\)\nb-algebra~\(A\) were a coproduct of
  \(\C\) and~\(\C\) in~\(\Csttwocat\).  Then the groupoid of arrows
  \(A\to\C\) would be equivalent to~\(\T^2\).  Its objects are
  non-zero characters \(A\to\C\) and its arrows are unitaries
  in~\(\C\) acting on characters by conjugation, that is, trivially.
  So we get a disjoint union of some copies of the group~\(\T\), one
  for each character of~\(A\).  But this is never equivalent
  to~\(\T^2\) because the groups \(\T\) and~\(\T^2\) are not
  isomorphic.  To see the latter, observe that~\(\T\) has exactly one
  element of order~\(2\), namely, \(-1\), while~\(\T^2\) has exactly
  three of them, namely, \((-1,+1)\), \((-1,-1)\), \((+1,-1)\).
\end{example}

The category~\(\Corrcat\) has more diagrams than~\(\Csttwocat\).
Proposition~\ref{pro:product_coproduct_in_Corr} and
Example~\ref{exa:coproduct_CC} show that some very simple diagrams
have a colimit in~\(\Corrcat\), but not in~\(\Csttwocat\).  In the
following, we therefore mostly study colimits in~\(\Corrcat\).

Next we clarify the role of free products in our theory.  We show that
amalgamated free products are pushouts in~\(\Corrcat\) under
a nondegeneracy assumption; this rules out, in particular, free
products without any amalgamation.  Indeed, in the most degenerate
case where we amalgamate over~\(0\),
Proposition~\ref{pro:product_coproduct_in_Corr} shows that the
coproduct is the \(\Cont_0\)\nb-direct sum and not the free product.

\subsection{Pushouts}
\label{sec:pushouts}

A \emph{pushout} in~\(\Corrcat_\prop\) is a colimit of a diagram of
the form
\[
\begin{tikzpicture}
  \node (A) at (0,0) {\(A\)};
  \node (B2) at (0,-1) {\(B_2,\)};
  \node (B1) at (1,0) {\(B_1\)};
  \draw[->] (A) to
  node[above] {\(\Hilm_1\)}  (B1);
  \draw[->] (A) to node[left] {\(\Hilm_2\)} (B2);
\end{tikzpicture}
\]
where \(A\), \(B_1\) and \(B_2\) are \(\Cst\)\nb-algebras and
\(\Hilm_1\) and~\(\Hilm_2\) are proper correspondences, without
further data or conditions.

One extreme case is \(A=0\), where the pushout degenerates to a
coproduct; this gives the direct sum \(B_1\oplus B_2\) by
Proposition~\ref{pro:product_coproduct_in_Corr}.  Here we consider
the opposite extreme case, where \(\Hilm_1\) and~\(\Hilm_2\) are
associated to \emph{nondegenerate} \Star{}homomorphisms \(A\to
B_1\), \(A\to B_2\); that is, \(\Hilm_i=B_i\) with~\(A\) acting by
\(a\cdot b\defeq \varphi_i(a)\cdot b\) for \(i=1,2\).

\begin{proposition}
  \label{pro:amalgamated_free_product}
  Let \(A\), \(B_1\) and \(B_2\) be \(\Cst\)\nb-algebras and let
  \(\varphi_1\colon A\to B_1\) and \(\varphi_2\colon A\to B_2\) be
  \emph{nondegenerate} \Star{}homomorphisms.  The amalgamated free
  product \(B_1\star_A B_2\) is also a pushout in~\(\Corrcat\).
\end{proposition}

\begin{proof}
  When we turn the \Star{}homomorphism~\(\varphi_i\) for \(i=1,2\)
  into a correspondence~\(\Hilm_i\), we take the right ideal
  \(\varphi_i(A)\cdot B_i\), viewed as a Hilbert \(B_i\)\nb-module,
  and equipped with the left action of~\(A\) through~\(\varphi_i\).
  Our nondegeneracy assumption means that \(\Hilm_i=B_i\) as a right
  Hilbert \(B_i\)\nb-module.  Furthermore, we remark that
  \(\varphi_i(A)\subseteq \Comp(\Hilm_i)=B_i\) by assumption, so
  the~\(\Hilm_i\) are proper correspondences.  We will see later that
  properness is crucial to get colimits.

  Let~\(D\) be a \(\Cst\)\nb-algebra.  A transformation
  in~\(\Corrcat\) from our pushout diagram to the constant diagram
  on~\(D\) is given by correspondences \(\Hilm[F]_1\colon B_1\to D\),
  \(\Hilm[F]_2\colon B_2\to D\) and an isomorphism
  \[
  U\colon
  \Hilm[F]_1 \cong B_1\otimes_{B_1} \Hilm[F]_1 \to
  B_2\otimes_{B_2} \Hilm[F]_2 \cong \Hilm[F]_2
  \]
  of correspondences from~\(A\) to~\(D\).  That is, \(U\) is a unitary
  operator \(\Hilm[F]_1\to\Hilm[F]_2\) that intertwines the left
  actions of~\(A\) given by composing the actions of \(B_i\) with the
  \Star{}homomorphisms~\(\varphi_i\).  Here we have used the
  nondegeneracy of~\(\varphi_i\) to identify \(\Hilm_i=B_i\) as
  Hilbert \(B_i\)\nb-modules.

  A modification from \((\Hilm[F]_i,U)\) to \((\Hilm[F]_i',U')\) is
  given by isomorphisms of correspondences \(V_i\colon
  \Hilm[F]_i\to\Hilm[F]'_i\) for \(i=1,2\) that intertwine \(U\)
  and~\(U'\).

  Every such transformation is isomorphic to one where
  \(\Hilm[F]_1=\Hilm[F]_2\) as right Hilbert \(D\)\nb-modules
  and~\(U\) is the identity operator: the identity on~\(\Hilm[F]_1\)
  and \(U\colon \Hilm[F]_1\to\Hilm[F]_2\) is an invertible
  modification.  Hence restricting to transformations with
  \(\Hilm[F]_1=\Hilm[F]_2\) and \(U=\Id\) gives an equivalent
  groupoid.  So it does not change the colimit.  The intertwining
  condition for modifications now simply says that the unitaries
  \(\Hilm[F]_i\to\Hilm[F]'_i\) for \(i=1,2\) are the \emph{same}
  unitary, so we only have a single unitary that intertwines the
  actions of \(B_1\) and~\(B_2\), and hence the actions of~\(A\).

  If \(\Hilm[F]_1=\Hilm[F]_2\) and \(U=\Id\), then \(B_1\)
  and~\(B_2\) act on the same Hilbert module, and the actions
  composed with~\(\varphi_i\) coincide on~\(A\); this gives an action
  of the amalgamated free product \(B_1\star_A B_2\)
  on~\(\Hilm[F]_i\).  Since \(B_1\) and~\(B_2\) act
  nondegenerately, so does \(B_1\star_A B_2\).  Hence we get a
  correspondence \(B_1\star_A B_2\to D\).

  Conversely, a correspondence \(B_1\star_A B_2\to D\) gives a Hilbert
  module~\(\Hilm[F]\) with a nondegenerate left action of~\(B_1\star_A
  B_2\).  Since \(A\cdot B_i=B_i\), the embedding \(A\to B_1\star_A
  B_2\) is nondegenerate, so the action of~\(A\) on~\(\Hilm[F]\) is
  nondegenerate, and then so are the actions of~\(B_i\).  Thus we get
  a transformation from the pushout diagram to the constant diagram
  on~\(D\) with \(\Hilm[F]=\Hilm[F]_1=\Hilm[F]_2\) and~\(U\) the
  identity.  Thus we have found an equivalence between the groupoid of
  natural transformations and modifications and the groupoid of
  correspondences from \(B_1\star_A B_2\) to~\(D\).  This proves that
  \(B_1\star_A B_2\) is a colimit.
\end{proof}

\begin{corollary}
  Let \(\Hilm_i\) be proper, full correspondences from~\(A\)
  to~\(B_i\) for \(i=1,2\).  The pushout in~\(\Corrcat\) of
  \(\Hilm_1\) and~\(\Hilm_2\) is the amalgamated free product
  \(\Comp(\Hilm_1)\star_A \Comp(\Hilm_2)\).
\end{corollary}

\begin{proof}
  Since~\(\Hilm_i\) is full, it gives a Morita--Rieffel equivalence
  between \(\Comp(\Hilm_i)\) and~\(B_i\).  Hence the diagrams
  in~\(\Corrcat\) given by \(\Hilm_1\) and~\(\Hilm_2\) and by the
  \Star{}\hspace{0pt}homomorphisms \(A\to\Comp(\Hilm_i)\) for \(i=1,2\) from the
  left \(A\)\nb-module structures on~\(\Hilm_i\) are isomorphic.  The
  latter diagram has \(\Comp(\Hilm_1)\star_A \Comp(\Hilm_2)\) as a
  colimit by Proposition~\ref{pro:amalgamated_free_product}.  Since
  the construction of colimits is functorial by
  Proposition~\ref{pro:colim_functorial}, this is also a colimit of
  the original diagram.
\end{proof}

\subsection{An example of a coequaliser}
\label{sec:coequalisers}

A \emph{coequaliser} is a colimit of a diagram consisting of two
parallel arrows
\(\alpha_1, \alpha_2 \colon A_1\rightrightarrows A_2\).
These particular colimits quickly become very complicated, as the
following example shows:

\begin{example}
  Consider the coequaliser of the following diagram:
  \begin{equation}
    \label{eq:coequaliser_example}
    \begin{tikzpicture}[scale=1,cdar,baseline=(current bounding box.west)]
      \node (1) at (0,0) {\(\C\)};
      \node (2) at (1,0) {\(\C\)};
      \draw[transform canvas={yshift=.5ex}] (1) -- node {\(\scriptstyle \C^m\)}
(2);
      \draw[transform canvas={yshift=-.1ex}] (1) -- node[swap] {\(\scriptstyle
\C^n\)} (2);
    \end{tikzpicture}
  \end{equation}
  The groupoid of transformations from the above diagram to the
  constant diagram on a \(\Cst\)\nb-algebra~\(D\) is equivalent to
  the groupoid that has pairs \((\Hilm[F],U)\) for a Hilbert
  \(D\)\nb-module~\(\Hilm[F]\) and a unitary operator
  \[
  U\colon \Hilm[F]^n \cong \C^n\otimes_{\C} \Hilm[F]\xrightarrow{\cong}
  \C^m\otimes_{\C} \Hilm[F]\cong \Hilm[F]^m
  \]
  as objects.  We may write~\(U\) as a matrix \(U=(u_{i,j})\) with
  \(u_{i,j}\in\Bound(\Hilm[F])\) for \(1\le i\le n\),
  \(1\le j \le m\).  The operator~\(U\) is unitary if and only if
  \begin{equation}
    \label{eq:coequaliser_example_relations}
    \sum_{k=1}^m u_{i_1 k} u_{i_2 k}^* = \delta_{i_1,i_2},\quad
    \sum_{k=1}^n u_{kj_1}^* u_{kj_2}=\delta_{j_1,j_2}
  \end{equation}
  for all \(1\le i_1,i_2\le n\), \(1\le j_1,j_2 \le m\).  Hence the
  universal \(\Cst\)\nb-algebra~\(U_{m\times n}^\mathrm{nc}\)
  generated by the elements~\(u_{ij}\) for \(i=1,\dotsc,n\),
  \(j=1,\dotsc,m\) that
  satisfy~\eqref{eq:coequaliser_example_relations} is a coequaliser
  of~\eqref{eq:coequaliser_example}.  For \(m=n\), this
  \(\Cst\)\nb-algebra is introduced by Lawrence
  Brown~\cite{Brown:Ext_free} and studied further by Kevin
  McClanahan, who showed that it has no projections
  (\cite{McClanahan:unitary_matrix}*{Corollary 2.7}) and is
  KK-equivalent to \(\Cst(\Z)\cong \Cont(\T)\)
  (\cite{McClanahan:KK_twisted}*{Proposition 5.5}).  The
  \(\Cst\)\nb-algebras~\(U_{m\times n}^\mathrm{nc}\) are
  \(\Cst\)\nb-algebra analogues of the algebras introduced by
  Leavitt~\cite{Leavitt:Module_type}, and they are prototypical
  examples of separated graph \(\Cst\)\nb-algebras
  (see~\cite{Ara-Goodearl:C-algebras_separated_graphs}).
\end{example}

\section{Colimits for group and crossed module actions}
\label{sec:groups_and_cm}

We now consider colimits where~\(\Cat\) is a group~\(G\) or a crossed
module.  We consider both target bicategories \(\Csttwocat\) and
\(\Corrcat\).  In all these cases, the identification of the colimit
with an appropriate ``crossed product'' is a mere reformulation of
results in \cites{Buss-Meyer-Zhu:Higher_twisted,
  Buss-Meyer:Crossed_products}.  Hence we will be rather brief.  These
results are trivial, but they are important motivation to look
at colimits in bicategories.

To make the results below look more surprising, we briefly consider
the colimit for a group action in the usual category of
\(\Cst\)\nb-algebras and \Star{}homomorphisms, without any
\(2\)\nb-arrows.  A group action by automorphisms is, indeed, the same
as a functor from~\(G\) to the category of \(\Cst\)\nb-algebras, given
by a \(\Cst\)\nb-algebra~\(A\) and \(\alpha_g\in\Aut(A)\) satisfying
\(\alpha_g \alpha_h = \alpha_{gh}\).  A cone over this diagram is a
\(\Cst\)\nb-algebra~\(B\) with a \Star{}homomorphism \(f\colon A\to
B\) such that \(f\circ\alpha_g=f\) for all \(g\in G\).  Thus~\(f\)
vanishes on the ideal~\(I_\alpha\) generated by \(\alpha_g(a)-a\) for
all \(g\in G\), \(a\in A\).  Indeed, the quotient map \(A\to
A/I_\alpha\) is \emph{the universal cone}.  Hence the colimit
is~\(A/I_\alpha\).  This is very often zero, and certainly not an
object worth studying.

When working in a bicategory, we replace the \emph{condition}
\(f\circ\alpha_g = f\) by \emph{extra data}, say, by a unitary~\(u_g\)
with \(u_g f(a) u_g^* = f(\alpha_g(a))\) for all \(a\in A\).  Thus the
bicategorical colimit is larger than~\(A\), very much
unlike~\(A/I_\alpha\) above.

The objects of~\(\Csttwocat^G\) are described concretely in
\cite{Buss-Meyer-Zhu:Higher_twisted}*{\S3.1.1} as Busby--Smith twisted
actions of~\(G\); those of~\(\Corrcat^G\) are equivalent to
saturated Fell bundles over~\(G\).  The transformations in
\(\Csttwocat^G\) and~\(\Corrcat^G\) are described concretely in
\cite{Buss-Meyer-Zhu:Higher_twisted}*{\S3.2}; modifications in
\(\Csttwocat^G\) and~\(\Corrcat^G\) are described concretely in
\cite{Buss-Meyer-Zhu:Higher_twisted}*{\S3.3}.  Results
in~\cite{Buss-Meyer-Zhu:Higher_twisted} immediately give the
following proposition:

\begin{proposition}
  \label{pro:colim_group_action_Csttwo}
  Let~\(G\) be a group.  Let \(\alpha\colon G\to\Aut(A)\) and
  \(\omega\colon G\times G\to\UM(A)\) be a Busby--Smith twisted
  action of~\(G\) on a \(\Cst\)\nb-algebra~\(A\).  The crossed
  product \(A\rtimes_{\alpha,\omega} G\) is a colimit of the functor
  \(F\colon G\to\Csttwocat\) associated to \((A,\alpha,\omega)\).
\end{proposition}

\begin{proof}
  Let~\(D\) be a \(\Cst\)\nb-algebra.  The functor \(\const_D\colon
  G\to\Csttwocat\) corresponds to the trivial action of~\(G\)
  on~\(D\).  A transformation from~\(F\) to \(\const_D\) is equivalent
  to a covariant representation of \((A,G,\alpha,\omega)\) in
  \(\Mult(D)\), that is, a nondegenerate representation
  \(\varrho\colon A\to\Mult(D)\) and a map \(\pi\colon G\to\UM(D)\)
  satisfying \(\pi_g\varrho(a)\pi_g^*= \varrho(\alpha_g(a))\) for all
  \(g\in G\), \(a\in A\) and
  \(\pi_{g_1}\pi_{g_2}=\varrho(\omega(g_1,g_2))\pi_{g_1g_2}\) for all
  \(g_1,g_2\in G\) (see \cite{Buss-Meyer-Zhu:Higher_twisted}*{Example
    3.8}).  Modifications between such transformations are the same as
  unitary equivalences of covariant representations by
  \cite{Buss-Meyer-Zhu:Higher_twisted}*{Example 3.13}.

  The crossed product is defined to be universal for covariant
  representations.  That is, there is a bijection between
  transformations from~\(F\) to \(\const_D\) and morphisms from
  \(A\rtimes_{\alpha,\omega} G\) to~\(D\); the modifications between
  the transformations corresponding to covariant representations
  \((\varrho,\pi)\) and \((\varrho',\pi')\) are unitary multipliers~\(u\)
  of~\(D\) with \(u\varrho(a)u^*=\varrho'(a)\) for all \(a\in A\) and
  \(u\pi(g)u^*=\pi'(g)\) for all \(g\in G\).  These are exactly the
  unitaries that intertwine the induced representations of
  \(A\rtimes_{\alpha,\omega} G\).  Thus the groupoids
  \(\Csttwocat^G(F,\const_D)\) and
  \(\Csttwocat(A\rtimes_{\alpha,\omega} G,D)\) are naturally
  isomorphic.
\end{proof}

For group actions by correspondences, that is, saturated Fell bundles,
the section \(\Cst\)\nb-algebra plays the role of the crossed
product:

\begin{proposition}
 \label{pro:colimit_Fell_bundle}
  Let~\(G\) be a group and let \((A_g)_{g\in G}\) be a saturated Fell
  bundle over~\(G\), viewed as a functor \(F\colon G\to \Corrcat\).  The
  section \(\Cst\)\nb-algebra of~\((A_g)_{g\in G}\) is a colimit
  of~\(F\).
\end{proposition}

\begin{proof}
  Let \(D\) be a \(\Cst\)\nb-algebra.  Then \(\const_D\) corresponds
  to the constant Fell bundle with fibres~\(D\), which describes the
  trivial action of~\(G\) on~\(D\).  Transformations to \(\const_D\)
  in~\(\Corrcat^G\) are in bijection with pairs~\((\Hilm,\pi)\),
  where~\(\Hilm\) is a Hilbert \(D\)\nb-module and \(\pi\colon
  \bigsqcup_{g\in G} A_g \to\Bound(\Hilm)\) is a nondegenerate Fell
  bundle representation (see the discussion before
  \cite{Buss-Meyer-Zhu:Higher_twisted}*{Definition 3.12}).
  Modifications between such transformations are equivalent to unitary
  intertwiners between Fell bundle representations.

  The section \(\Cst\)\nb-algebra \(C\defeq \Cst(A_g)_{g\in G}\) is
  defined as the \(\Cst\)\nb-completion of the convolution algebra of
  sections of the Fell bundle.  By definition, representations of a
  Fell bundle integrate to \Star{}representations of this section
  \(\Cst\)\nb-algebra, and all representations of~\(C\) come from Fell
  bundle representations.  Furthermore, a Fell bundle representation
  is nondegenerate if and only if the resulting representation
  of~\(C\) is nondegenerate.  A nondegenerate representation of~\(C\)
  on a Hilbert \(D\)\nb-module is the same as a correspondence
  from~\(C\) to~\(D\).  Furthermore, an operator intertwines the Fell
  bundle representations if and only if it intertwines the resulting
  representations of~\(C\), that is, is an isomorphism of
  correspondences.  Hence the groupoids \(\Corrcat(F,\const_D)\) and
  \(\Corrcat(C,D)\) are naturally isomorphic.
\end{proof}

Summing up, we merely have to inspect the description of
transformations and modifications between functors \(G\to\Csttwocat\)
or \(G\to\Corrcat\) in~\cite{Buss-Meyer-Zhu:Higher_twisted} to see
that the colimit in either case is the crossed product or Fell bundle
section algebra, respectively.

\smallskip

Now let~\(\cm\) be a \emph{crossed module}; that is, \(\cm\) consists
of two groups \(G\) and~\(H\) with homomorphisms \(\partial\colon H\to
G\) and \(c\colon G\to\Aut(H)\), such that \(\partial(c_g(h)) =
g\partial(h) g^{-1}\) and \(c_{\partial h}(k) = hkh^{-1}\) for all
\(g\in G\), \(h,k\in H\).

Strict actions of crossed modules on \(\Cst\)\nb-algebras and crossed
products for such actions are defined
in~\cite{Buss-Meyer-Zhu:Non-Hausdorff_symmetries}.  These are more
special than functors \(\cm\to\Csttwocat\), which are discussed in
\cite{Buss-Meyer-Zhu:Higher_twisted}*{\S4.1.1}.  Functors
\(\cm\to\Corrcat\) are described in
\cite{Buss-Meyer:Crossed_products}*{Theorem 2.11}, generalising the
notion of a saturated Fell bundle from groups to crossed modules.  The
crossed product for a functor \(F\colon \cm\to\Corrcat\) is defined in
\cite{Buss-Meyer:Crossed_products}*{Definition 2.8} by a universal
property and identified more concretely in
\cite{Buss-Meyer:Crossed_products}*{Proposition 2.17}.

\begin{proposition}
  \label{pro:colim_cm_action_Csttwo}
  The crossed product for a crossed module action by correspondences
  is a colimit in~\(\Corrcat\).
\end{proposition}

\begin{proof}
  Let \(F\colon \cm\to\Corrcat\) be a functor.  As in the group case,
  the proof is by making explicit what transformations
  \(F\to\const_D\) and modifications between them are and observing
  that the resulting universal property for the colimit is the same
  one as the defining universal property of the crossed product.
  Since this is routine checking, we omit further details.
\end{proof}

\section{A single endomorphism}
\label{sec:single_endo}

Before we study colimits of arbitrary shape, we look at an important
special case: let~\(\Cat\) be the monoid~\((\N,+)\), viewed as a
category with a single object.

A functor \(\Cat\to\Corrcat\) is given by a \(\Cst\)\nb-algebra~\(A\),
correspondences \(\Hilm_n\colon A\to A\) for \(n\in\N\) and
isomorphisms of correspondences \(\mu_{n,m}\colon \Hilm_n\otimes_A
\Hilm_m \cong \Hilm_{n+m}\) for all \(n,m\in\N\), such
that~\(\Hilm_0\) is the identity correspondence, \(\mu_{0,m}\)
and~\(\mu_{n,0}\) are the canonical transformations, and the
multiplication maps~\(\mu_{n,m}\) are associative in a suitable sense.
This is a special case of
Proposition~\ref{pro:functor_category-diagram} below.

A transformation between such diagrams \((A,\Hilm_n,\mu_{n,m})\),
\((B,\Hilm[F]_n,v_{n,m})\), is given by a correspondence
\(\Hilm[G]\colon A\to B\) and isomorphisms
\begin{equation}
  \label{eq:endo_trafo_w_n}
  \Hilm_n\otimes_A \Hilm[G] \xrightarrow[\cong]{w_n}
  \Hilm[G]\otimes_B \Hilm[F]_n
\end{equation}
for all \(n\in\N\), subject to compatibility conditions with the
\(\mu_{n,m}\) and~\(v_{n,m}\) for \(n,m\in\N\) and the condition
that~\(w_0\) should be the canonical isomorphism (see
Proposition~\ref{pro:trafo_category-diagram}).  A modification
between two such transformations, \((\Hilm[G],w_n)\) and
\((\Hilm[G]',w_n')\), is given by an isomorphism of correspondences
\(\Hilm[G]\to\Hilm[G]'\) intertwining the \(w_n\) and~\(w_n'\) in
the obvious sense (see also
Proposition~\ref{pro:modification_category-diagram}).

This data can be simplified because the monoid~\((\N,+)\) is freely
generated by \(1\in\N\).  For a functor \(\N\to\Corrcat\), it is
enough to give~\(A\) and a single correspondence \(\Hilm=\Hilm_1\),
with no further data or conditions.  We may extend this to a functor
in the above sense by letting \(\Hilm_n\defeq \Hilm^{\otimes_A n}\)
for \(n\in\N\) (understood to be the identity correspondence if
\(n=0\)), and letting~\(\mu_{n,m}\) be the canonical map (this is the
identity map up to the associators, which we have dropped from our
notation).  The conditions on the~\(\mu_{n,m}\) ensure that any
functor is isomorphic to one of this form.

Next, a transformation is specified by a correspondence~\(\Hilm[G]\)
and an isomorphism
\[
w=w_1\colon \Hilm \otimes_A \Hilm[G] \cong \Hilm[G]\otimes_B
\Hilm[F],
\]
with no condition on~\(w\): iteration of~\(w_1\) provides the
isomorphisms~\(w_n\) for \(n\in\N\) as in~\eqref{eq:endo_trafo_w_n},
and the compatibility conditions for the~\(w_n\) say that any
transformation is generated from~\(w_1\) in this way.  Finally, for a
modification, it is enough to require the intertwining condition
for~\(w_1\), then the condition follows for~\(w_n\) for all
\(n\in\N\).  In brief, the bicategory of functors \(\N\to\Corrcat\) is
equivalent to the following simpler bicategory:
\begin{enumerate}
\item objects are given by a \(\Cst\)\nb-algebra~\(A\) and a
  correspondence \(\Hilm\to\Hilm\);
\item arrows \((A,\Hilm)\to (B,\Hilm[F])\) are given by a
  correspondence \(\Hilm[G]\colon A\to B\) and an isomorphism of
  correspondences \(w\colon \Hilm[G]\otimes_B \Hilm[F]\cong
  \Hilm\otimes_A \Hilm[G]\);
\item \(2\)\nb-arrows \((\Hilm[G],w)\to (\Hilm[G]',w')\) are
  isomorphisms \(x\colon \Hilm[G]\to\Hilm[G]'\) such that
  \((\Id_{\Hilm}\otimes_A x)\circ w = w\circ (x\otimes_B
  \Id_{\Hilm[F]})\).
\end{enumerate}
We may use the simplified data to describe colimits as well, which
only require equivalences of categories.

We now analyse transformations from~\((A,\Hilm)\) to a constant
diagram \(\const_D\).  First, \(\const_D = (D,D)\), where the
second~\(D\) means the identity correspondence on~\(D\).  Hence the
isomorphism~\(w\) in a transformation may also be viewed as an
isomorphism \(\Hilm[G]\cong \Hilm\otimes_A \Hilm[G]\); here we use the
canonical isomorphism \(\Hilm[G]\otimes_D D\cong \Hilm[G]\).

Roughly speaking, we want to turn an isomorphism \(w\colon
\Hilm[G]\congto \Hilm\otimes_A \Hilm[G]\) into a representation of a
\(\Cst\)\nb-algebra on~\(\Hilm[G]\).  The necessary work is carried
out in~\cite{Albandik-Meyer:Product}.  First, the isomorphism
\(w\colon \Hilm[G]\congto \Hilm\otimes_A \Hilm[G]\) is turned into a
``representation'' \(\Hilm\to \Bound(\Hilm[G])\) by sending
\(\xi\in\Hilm\) to the operator
\[
\Hilm[G]\ni\eta \mapsto w^*(\xi\otimes\eta)\in \Hilm[G].
\]
This is a representation of the Hilbert module~\(\Hilm\) in the
standard sense, satisfying an extra nondegeneracy condition
corresponding to the surjectivity of~\(w^*\).  This extra
nondegeneracy condition is equivalent to the Cuntz--Pimsner covariance
condition \emph{provided~\(\Hilm\) is a proper correspondence} by
\cite{Albandik-Meyer:Product}*{Proposition 2.5}.  This leads to the
following theorem:

\begin{theorem}
  \label{the:CP_as_colimit}
  Let \(\Hilm\colon A\to A\) be a proper correspondence.  The
  Cuntz--Pimsner algebra of~\(\Hilm\) is a colimit of the
  corresponding diagrams \((\N,+)\to\Corrcat_\prop\) and
  \((\N,+)\to\Corrcat\).
\end{theorem}

\begin{proof}
  The Cuntz--Pimsner algebra~\(\CP_{\Hilm}\) is characterised by the
  universal property that \Star{}homomorphisms \(\CP_{\Hilm}\to D\)
  for a \(\Cst\)\nb-algebra~\(D\) are in bijection with pairs
  \((\varphi,\vartheta)\), where \(\varphi\colon A\to D\) is a
  \Star{}homomorphism and \(\vartheta\colon \Hilm\to D\) is a
  Cuntz--Pimsner covariant representation of~\(\Hilm\) (see
  \cite{Pimsner:Generalizing_Cuntz-Krieger}*{Theorem 3.12}).  In
  particular, \(A\subseteq\CP_{\Hilm}\), and inspection shows that
  this embedding is nondegenerate, that is, \(A\cdot \CP_{\Hilm}\) is
  dense in~\(\CP_{\Hilm}\).  It follows that the \Star{}homomorphism
  \(\CP_{\Hilm}\to\Bound(\Hilm[F])\) associated to \(\varphi\colon
  A\to\Bound(\Hilm[F])\) and \(\vartheta\colon
  \Hilm\to\Bound(\Hilm[F])\) is nondegenerate if and only
  if~\(\varphi\) is nondegenerate.  Thus a correspondence
  from~\(\CP_{\Hilm}\) to~\(D\) is the same as a correspondence
  \((\Hilm[F],\varphi)\) from~\(A\) to~\(D\) with a map
  \(\vartheta\colon \Hilm\to\Bound(\Hilm[F])\) which, together
  with~\(\varphi\), is a Cuntz--Pimsner covariant representation.

  Since~\(\Hilm\) is proper, the Cuntz--Pimsner covariance condition
  for~\(\vartheta\) holds if and only if~\(\vartheta\) is
  ``nondegenerate'' in the sense that the closed linear span of
  \(\vartheta(\Hilm)\cdot (\Hilm[F])\) is~\(\Hilm[F]\) (see
  \cite{Albandik-Meyer:Product}*{Proposition 2.5}).  Such
  nondegenerate correspondences are in bijection with isomorphisms of
  correspondences \(\Hilm\otimes\Hilm[F]\cong\Hilm[F]\) by
  \cite{Albandik-Meyer:Product}*{Proposition 2.3}.  So a
  correspondence from~\(\CP_{\Hilm}\) to~\(D\) is the same as a
  correspondence~\(\Hilm[F]\) from~\(A\) to~\(D\) with an isomorphism
  of correspondences \(\Hilm\otimes_A \Hilm[F]\cong\Hilm[F]\).  These
  are exactly the simplified transformations of functors
  \((\N,+)\to\Corrcat\), by the discussion above the theorem.

  Isomorphisms of correspondences \(\CP_{\Hilm}\to D\) are the same as
  unitaries \(\Hilm[F]\to\Hilm[F]'\) that intertwine the left actions
  of \(A\) and~\(\Hilm\).  Intertwining the left actions of~\(A\)
  means that they are isomorphisms of correspondences from~\(A\)
  to~\(D\), and intertwining the representations of~\(\Hilm\) means
  that they are modifications between the corresponding
  transformations of functors \((\N,+)\to\Corrcat\).  Hence the
  groupoid of correspondences \(\CP_{\Hilm}\to D\) and their
  isomorphisms is naturally isomorphic to the groupoid of simplified
  transformations \((A,\Hilm)\to\const_D\) and their modifications.
  This says that~\(\CP_{\Hilm}\) has the universal property of a
  colimit in~\(\Corrcat\).

  A correspondence \(\Hilm[F]\colon \CP_{\Hilm}\to D\) is proper if
  and only if the corresponding representation~\(\varphi\) of~\(A\)
  has image in the compact operators,
  \(\varphi(A)\subseteq\Comp(\Hilm[F])\); this is because \(A\cdot
  \CP_{\Hilm} = \CP_{\Hilm}\).  Hence~\(\CP_{\Hilm}\) has the
  universal property of a colimit in~\(\Corrcat_\prop\) as well.
\end{proof}

Note that the colimit is the Cuntz--Pimsner algebra right away, the
Cuntz--Toeplitz algebra plays no role; this is because of the
built-in nondegeneracy properties of~\(\Corrcat\).

Following Muhly and Solel~\cite{Muhly-Solel:Tensor} and
Katsura~\cite{Katsura:Cstar_correspondences}, many authors have
modified the definition of the Cuntz--Pimsner algebra by requiring the
Cuntz--Pimsner covariance condition only on a suitable ideal in
\(\varphi_{\Hilm}^{-1}(\Comp({\Hilm}))\).  Such modifications are
particularly popular if the left action of~\(A\) on~\(\Hilm\) is not
faithful because in that case, the unmodified Cuntz--Pimsner algebra
may well be zero.  The colimit construction, however, singles out the
unmodified Cuntz--Pimsner algebra.

Unlike the Cuntz--Pimsner condition, ``nondegeneracy'' is not a
relation that we may impose on a bunch of generators.  This is why
there need not be a universal \(\Cst\)\nb-algebra for nondegenerate
representations, but there is always one for Cuntz--Pimsner covariant
representations.  The two properties are only equivalent if~\(\Hilm\)
is proper.  This is the reason why we only understand colimits for
diagrams of proper correspondences.

It seems likely that the colimit of the diagram
\((\N,+)\to\Corrcat\) given by the endomorphism \(\ell^2(\N)\)
of~\(\C\) does not exist (see \cite{Albandik-Meyer:Product}*{Example
  2.7}).  In the following, we therefore restrict attention to
colimits of diagrams of \emph{proper} correspondences.

\section{Category-shaped diagrams and product systems}
\label{sec:category-diagrams}

We have examined enough examples that it makes sense to spell out what
functors, transformations, and modifications \(\Cat\to\Corrcat\) mean
for an arbitrary category~\(\Cat\).  We are particularly interested in
transformations to a constant functor, which lead to the descrption of
the colimit of a diagram.

\subsection{Functors, transformations and modifications}
\label{sec:Corrcat_Cat}

The objects of~\(\Corrcat^{\Cat}\) are functors (or morphisms)
\(\Cat\to\Corrcat\);
arrows are transformations between such functors, and \(2\)\nb-arrows
are modifications.  We describe these things more concretely and then
explain briefly how to compose transformations.  These definitions are
summarised succinctly in~\cite{Leinster:Basic_Bicategories}.  They are
worked out for~\(\Csttwocat^{\Cat}\) in
\cite{Buss-Meyer-Zhu:Higher_twisted}*{\S4}, even for an arbitrary
bicategory~\(\Cat\).  The definitions simplify if~\(\Cat\) is a
category because part of the data does not occur any more.  The
following propositions already contain these simplifications.  We omit
the (rather trivial) proofs.  Readers that do not care much about
bicategory theory could take the following propositions as
definitions.

\begin{proposition}
  \label{pro:functor_category-diagram}
  A functor \(\Cat\to\Corrcat\) consists of
  \begin{itemize}
  \item \(\Cst\)\nb-algebras~\(A_x\) for all objects~\(x\)
    of~\(\Cat\);
  \item correspondences \(\Hilm_g\colon A_x\to A_y\) for all arrows
    \(g\colon x\to y\) in~\(\Cat\);
  \item isomorphisms of correspondences \(\mu_{g,h}\colon
    \Hilm_h\otimes_{A_y} \Hilm_g\to \Hilm_{gh}\) for all pairs of
    composable arrows \(g\colon y\to z\), \(h\colon x\to y\)
    in~\(\Cat\);
  \end{itemize}
  such that
  \begin{enumerate}
  \item \(\Hilm_{1_x}\) is the identity correspondence on~\(A_x\) for
    all objects \(x\) of~\(\Cat\);
  \item \(\mu_{1_y,g}\colon \Hilm_g \otimes_{A_y} A_y \to \Hilm_g\)
    and~\(\mu_{g,1_x}\colon A_x \otimes_{A_x} \Hilm_g \to \Hilm_g\) are
    the canonical isomorphisms for all arrows \(g\colon x\to y\)
    in~\(\Cat\);
  \item for all composable arrows \(g_{01}\colon x_0\to x_1\),
    \(g_{12}\colon x_1\to x_2\), \(g_{23}\colon x_2\to x_3\), the
    following diagram commutes:
    \begin{equation}
      \label{eq:coherence_category-diagram}
      \begin{tikzpicture}[yscale=1,xscale=4,baseline=(current bounding box.west)]
        \node (m-1-1) at (144:1) {\((\Hilm_{g_{01}}\otimes_{A_{x_1}} \Hilm_{g_{12}})
          \otimes_{A_{x_2}} \Hilm_{g_{23}}\)};
        \node (m-1-1b) at (216:1) {\(\Hilm_{g_{01}}\otimes_{A_{x_1}}
          (\Hilm_{g_{12}}\otimes_{A_{x_2}} \Hilm_{g_{23}})\)};
        \node (m-1-2) at (72:1) {\(\Hilm_{g_{02}}\otimes_{A_{x_2}}\Hilm_{g_{23}}\)};
        \node (m-2-1) at (288:1) {\(\Hilm_{g_{01}}\otimes_{A_{x_1}}\Hilm_{g_{13}}\)};
        \node (m-2-2) at (0:.8) {\(\Hilm_{g_{03}}\)};

        \draw[dar] (m-1-1) -- node {\textup{can.}} (m-1-1b);
        \draw[dar] (m-1-1.north) -- node {\(\mu_{g_{12},g_{01}}\otimes_{A_{x_2}}\Id_{\Hilm_{g_{23}}}\)} (m-1-2);
        \draw[dar] (m-1-1b.south) -- node[swap] {\( \Id_{\Hilm_{g_{01}}}\otimes_{A_{x_1}}\mu_{g_{23},g_{12}}\)} (m-2-1);
        \draw[dar] (m-1-2.south) -- node[inner sep=0pt] {\(\mu_{g_{23},g_{02}}\)} (m-2-2);
        \draw[dar] (m-2-1.north) -- node[swap,inner sep=1pt] {\(\mu_{g_{13},g_{01}}\)} (m-2-2);
      \end{tikzpicture}
    \end{equation}
    Here \(g_{02}\defeq g_{12}\circ g_{01}\), \(g_{13}\defeq
    g_{23}\circ g_{12}\), and \(g_{03}\defeq g_{23}\circ g_{12}\circ
    g_{01}\).
  \end{enumerate}
  The diagram~\eqref{eq:coherence_category-diagram} commutes
  automatically if one of the arrows \(g_{01}\), \(g_{12}\)
  or~\(g_{23}\) is an identity arrow.
\end{proposition}

\begin{proposition}
  \label{pro:trafo_category-diagram}
  Let \((A_x^0,\Hilm_g^0,\mu_{g,h}^0)\) and
  \((A_x^1,\Hilm_g^1,\mu_{g,h}^1)\) be two functors from~\(\Cat\)
  to~\(\Corrcat\).  A transformation between them consists of
  \begin{itemize}
  \item correspondences~\(\gamma_x\) from~\(A_x^0\) to \(A_x^1\) for
    all objects~\(x\) of~\(\Cat\);
  \item isomorphisms of correspondences \(V_g\colon
    \gamma_x\otimes_{A^1_x} \Hilm_g^1\to \Hilm_g^0\otimes_{A^0_y}
    \gamma_y\) for all arrows \(g\colon x\to y\) in~\(\Cat\);
  \end{itemize}
  such that
  \begin{enumerate}
  \item \(V_{1_x}\colon \gamma_x\otimes_{A^1_x} A^1_x\to A^0_x
    \otimes_{A^0_x} \gamma_x\) is the canonical isomorphism
    through~\(\gamma_x\) for each object~\(x\) in~\(\Cat\);
  \item for each pair of composable arrows \(g\colon y\to z\),
    \(h\colon x\to y\) in~\(\Cat\), the following diagram commutes:
    \begin{equation}
      \label{eq:trafo_category-diagram}
      \begin{tikzpicture}[yscale=1.5,xscale=2.5,baseline=(current bounding box.west)]
        \node (cbb) at (144:1) {\(\gamma_x\otimes_{A^1_x} \Hilm^1_h\otimes_{A^1_y} \Hilm^1_g\)};
        \node (cb) at (216:1) {\(\gamma_x\otimes_{A^1_x} \Hilm^1_{gh}\)};
        \node (acb) at (72:1) {\(\Hilm^0_h\otimes_{A^0_y} \gamma_y \otimes_{A^1_y} \Hilm^1_g\)};
        \node (ac) at (288:1) {\(\Hilm^0_{gh}\otimes_{A^0_z} \gamma_z\)};
        \node (aac) at (0:.8) {\(\Hilm^0_h\otimes_{A^0_y} \Hilm^0_g\otimes_{A^0_z} \gamma_z\)};
        \draw[dar] (cbb) -- node[swap] {\(\Id_{\gamma_x}\otimes_{A^1_x} u^1_{g,h}\)} (cb);
        \draw[dar] (cbb.north) -- node[near end] {\(V_h\otimes_{A^1_y} \Id_{\Hilm^1_g}\)} (acb);
        \draw[dar] (cb.south) -- node[swap,near end] {\(V_{gh}\)} (ac);
        \draw[dar] (acb.south) -- node[inner sep=0pt] {\(\Id_{\Hilm^0_h}\otimes_{A^0_y} V_g\)} (aac.north);
        \draw[dar] (aac.south) -- node {\(u^0_{g,h} \otimes_{A^0_z} \Id_{\gamma_z}\)} (ac.north);
      \end{tikzpicture}
    \end{equation}
  \end{enumerate}
  The diagram~\eqref{eq:trafo_category-diagram} commutes
  automatically if \(g\) or~\(h\) is an identity arrow.
\end{proposition}

\begin{proposition}
  \label{pro:modification_category-diagram}
  Let \((A_x^0,\Hilm_g^0,\mu_{g,h}^0)\) and
  \((A_x^1,\Hilm_g^1,\mu_{g,h}^1)\) be functors from~\(\Cat\)
  to~\(\Corrcat\) and let \((\gamma^1_x,V^1_g)\) and
  \((\gamma^2_x,V^2_g)\) be transformations between them.  A
  modification from \((\gamma^1_x,V^1_g)\) to \((\gamma^2_x,V^2_g)\)
  consists of isomorphisms of correspondences \(W_x\colon \gamma^1_x
  \to \gamma^2_x\) for all objects~\(x\) in~\(\Cat\) such that the
  diagrams
  \begin{equation}
    \label{eq:modification_category-diagram}
    \begin{tikzpicture}[yscale=1,xscale=4,baseline=(current bounding box.west)]
      \node (add) at (0,1) {\(\gamma^1_x \otimes_{A^1_x}\Hilm^1_g\)};
      \node (ad) at (1,1) {\(\gamma^2_x\otimes_{A^1_x} \Hilm^1_g\)};
      \node (ddc) at (0,0) {\(\Hilm^0_g\otimes_{A^0_y} \gamma^1_y\)};
      \node (dc) at (1,0) {\(\Hilm^0_g\otimes_{A^0_y} \gamma^2_y\)};

      \draw[dar] (add) -- node {\(W_x\otimes_{A^1_x} \Id_{\Hilm^1_g}\)} (ad);
      \draw[dar] (add) -- node[swap] {\(V^1_g\)} (ddc);
      \draw[dar] (ad) -- node {\(V^2_g\)} (dc);
      \draw[dar] (ddc) -- node[swap] {\(\Id_{\Hilm^0_g}\otimes_{A^0_y} W_y\)} (dc);
    \end{tikzpicture}
  \end{equation}
  commute for all arrows \(g\colon x\to y\) in~\(\Cat\).  This diagram
  commutes automatically if~\(g\) is an identity arrow.
\end{proposition}

The composition of transformations is defined as follows.  Describe
functors \(\Cat\to \Corrcat\) by \((A_x^0,\Hilm_g^0,\mu_{g,h}^0)\),
\((A_x^1,\Hilm_g^1,\mu_{g,h}^1)\) and
\((A_x^2,\Hilm_g^2,\mu_{g,h}^2)\), and transformations between them by
\((\gamma^{01}_x,V^{01}_g)\) and \((\gamma^{12}_x,V^{12}_g)\) as
above.  The product is given by the correspondences \(\gamma^{02}_x
\defeq \gamma^{01}_x \otimes_{A_x^1} \gamma^{12}_x\) from~\(A_x^0\)
to~\(A_x^2\) for objects~\(x\) of~\(\Cat\) and by the isomorphisms of
correspondences
\begin{multline*}
  V^{02}_g\colon \gamma_x^{02} \otimes_{A_x^2} \Hilm_g^2 =
  \gamma_x^{01} \otimes_{A_x^1} \gamma_x^{12} \otimes_{A_x^2} \Hilm_g^2
  \xrightarrow{\Id_{\gamma_x^{01}} \otimes_{A_x^1} V^{12}_g}
  \gamma_x^{01} \otimes_{A_x^1} \Hilm_g^1 \otimes_{A_y^2} \gamma_y^{12}
  \\\xrightarrow{V^{01}_g\otimes_{A_y^2} \Id_{\gamma_y^{12}}}
  \Hilm_g^0 \otimes_{A_y^0} \gamma_y^{01} \otimes_{A_y^1}  \gamma_y^{12}
  = \Hilm_g^0 \otimes_{A_y^0} \gamma_y^{02}
\end{multline*}
for arrows \(g\colon x\to y\) in~\(\Cat\).  These
\((\gamma^{02}_x,V^{02}_g)\) indeed form a transformation.  General
bicategory theory predicts that this composition turns
\(\Corrcat^{\Cat}\) into a bicategory again, and this is routine to
check by hand.

To understand the above definitions, consider the special case
where~\(\Cat\) has only one object, that is, \(\Cat\) is a monoid.
Then we may drop all indices~\(x\) above: a functor provides a single
\(\Cst\)\nb-algebra~\(A\), a transformation a single
correspondence~\(\gamma\), and a modification a single
isomorphism~\(W\).  Furthermore, all arrows in~\(\Cat\) are
composable, and there is only one identity morphism.  Simplifying the
data in Proposition~\ref{pro:functor_category-diagram} accordingly,
the result is very close to a \emph{product system} in the notation of
Fowler~\cite{Fowler:Product_systems}.

There are only two differences.  First, we require all left actions on
Hilbert modules to be nondegenerate (or ``essential''), whereas Fowler
is careful to avoid this assumption.  Secondly, we multiply in the
opposite order, \(\Hilm_h\otimes_A \Hilm_g\to\Hilm_{gh}\), which
corresponds to the composition of \Star{}homomorphisms.  As a result,
functors \(M\to\Corrcat\) for a monoid~\(M\) are the same as essential
product systems over the opposite monoid~\(M^\op\).

When we pass from monoids to categories, the only change is that we
get more than one \(\Cst\)\nb-algebra: one for each object of the
category.

Nondegeneracy of the left actions on correspondences is necessary for
unit arrows in~\(\Corrcat\) to work as expected: otherwise we would
not get a bicategory.  The order reversal comes in because when we
pass from \Star{}homomorphisms to correspondences, the composition of
\Star{}homomorphisms becomes the reverse-order tensor product.  With
our convention, monoid actions by \Star{}endomorphisms become actions
by correspondences of the same monoid.  The same order-reversal also
appears when translating between actions of a group by correspondences
and saturated Fell bundles over the group.  It is the reason
why~\(g^{-1}\) appears in the correspondence between functors
\(G\to\Corrcat\) and saturated Fell bundles over~\(G\) in the proof of
\cite{Buss-Meyer-Zhu:Higher_twisted}*{Theorem 3.3}.

\subsection{Colimits}
\label{sec:colimits_category-shaped}

Let~\(\Cat\) be a category, let \((A_x,\Hilm_g,\mu_{g,h})\) describe a
functor \(F\colon \Cat\to\Corrcat\) as in
Proposition~\ref{pro:functor_category-diagram}, and let~\(D\) be a
\(\Cst\)\nb-algebra.  We first describe the constant functor
\(\const_D\colon \Cat\to\Corrcat\).  Then we specialise the
description of transformations and modifications to the case of a
constant target.  We use this to describe the colimit of a proper
product system by generators and relations.

\begin{definition}
  \label{def:const_D}
  Let~\(D\) be a \(\Cst\)\nb-algebra.  The \emph{constant functor}
  \(\const_D\colon \Cat\to\Corrcat\) maps all objects~\(x\)
  of~\(\Cat\) to~\(D\), all arrows~\(g\) in~\(\Cat\) to the identity
  correspondence on~\(D\), and all pairs \(g,h\) to the canonical
  isomorphism \(D\otimes_D D\to D\).
\end{definition}

A transformation from the functor given by \((A_x,\Hilm_g,\mu_{g,h})\)
to \(\const_D\) is given by correspondences~\(\gamma_x\)
from~\(A_x\) to~\(D\) for all objects~\(x\) of~\(\Cat\) and
isomorphisms of correspondences
\[
V_g\colon \gamma_x\to\Hilm_g\otimes_{A_y} \gamma_y
\qquad\text{for all arrows }g\colon x\to y\text{ in }\Cat,
\]
such that \(V_{1_x}\) for an object~\(x\) is the canonical isomorphism
and the diagrams
\[
\begin{tikzpicture}[yscale=1.5,xscale=2.5,baseline=(current bounding box.west)]
  \node (cbb) at (144:1) {\(\gamma_x\)};
  \node (cb) at (216:1) {\(\gamma_x\)};
  \node (acb) at (72:1) {\(\Hilm_h\otimes_{A_y} \gamma_y\)};
  \node (ac) at (288:1) {\(\Hilm_{gh}\otimes_{A_z} \gamma_z\)};
  \node (aac) at (0:.8) {\(\Hilm_h\otimes_{A_y} \Hilm_g\otimes_{A_z} \gamma_z\)};
  \draw[dar] (cbb) -- node[swap] {\(\Id_{\gamma_x}\)} (cb);
  \draw[dar] (cbb) -- node[near end] {\(V_h\)} (acb);
  \draw[dar] (cb) -- node[swap,near end] {\(V_{gh}\)} (ac);
  \draw[dar] (acb.south) -- node {\(\Id_{\Hilm_h}\otimes_{A_y} V_g\)} (aac.north);
  \draw[dar] (aac.south) -- node {\(\mu_{g,h} \otimes_{A_z} \Id_{\gamma_z}\)} (ac.north);
\end{tikzpicture}
\]
for composable arrows \(g\colon y\to z\), \(h\colon x\to y\)
in~\(\Cat\) commute.  Here we simplified the data in
Proposition~\ref{pro:trafo_category-diagram} using the canonical
isomorphisms \(\gamma_x\otimes_D D\cong\gamma_x\) for all~\(x\); we
may, of course, drop the identity arrow on~\(\gamma_x\) and redraw
this diagram as a commuting square:
\begin{equation}
  \label{eq:cone_category-diagram}
  \begin{tikzpicture}[yscale=1,xscale=4.5,baseline=(current bounding box.west)]
    \node (cbb) at (0,1) {\(\gamma_x\)};
    \node (acb) at (1,1) {\(\Hilm_h\otimes_{A_y} \gamma_y\)};
    \node (ac) at (0,0) {\(\Hilm_{gh}\otimes_{A_z} \gamma_z\)};
    \node (aac) at (1,0) {\(\Hilm_h\otimes_{A_y} \Hilm_g\otimes_{A_z} \gamma_z\)};

    \draw[dar] (cbb) -- node {\(V_h\)} (acb);
    \draw[dar] (cbb) -- node[swap] {\(V_{gh}\)} (ac);
    \draw[dar] (acb) -- node {\(\Id_{\Hilm_h}\otimes_{A_y} V_g\)} (aac);
    \draw[dar] (aac) -- node {\(\mu_{g,h} \otimes_{A_z} \Id_{\gamma_z}\)} (ac);
  \end{tikzpicture}
\end{equation}
This diagram commutes automatically if \(g\) or~\(h\) is an identity
arrow.

If \((\gamma_x^1,V_g^1)\) and \((\gamma_x^2,V_g^2)\) are two such
transformations, then a modification between them is given by
isomorphisms of correspondences
\[
W_x\colon \gamma_x^1\to \gamma_x^2
\qquad\text{for all objects~\(x\) of~\(\Cat\),}
\]
such that the diagrams
\begin{equation}
  \label{eq:cone_modification_category-diagram}
  \begin{tikzpicture}[yscale=1,xscale=4,baseline=(current bounding box.west)]
    \node (add) at (0,1) {\(\gamma^1_x\)};
    \node (ad) at (1,1) {\(\gamma^2_x\)};
    \node (ddc) at (0,0) {\(\Hilm_g\otimes_{A_y} \gamma^1_y\)};
    \node (dc) at (1,0) {\(\Hilm_g\otimes_{A_y} \gamma^2_y\)};

    \draw[dar] (add) -- node {\(W_x\)} (ad);
    \draw[dar] (add) -- node[swap] {\(V^1_g\)} (ddc);
    \draw[dar] (ad) -- node {\(V^2_g\)} (dc);
    \draw[dar] (ddc) -- node[swap] {\(\Id_{\Hilm_g}\otimes_{A_y} W_y\)} (dc);
  \end{tikzpicture}
\end{equation}
commute for all arrows \(g\colon x\to y\) in~\(\Cat\).  This diagram
commutes automatically if~\(g\) is an identity arrow.

The colimit for a functor \(F\colon \Cat\to\Corrcat\) is, by
definition, a \(\Cst\)\nb-algebra~\(B\) such that, for each
\(\Cst\)\nb-algebra~\(D\), the groupoid of correspondences \(B\to D\)
and isomorphisms of correspondences between them is naturally
equivalent to the groupoid of transformations \(F\to\const_D\) and
modifications between them.

\begin{proposition}
  \label{pro:cone_as_representation}
  There is a bijection between transformations \(F\to\const_D\) and
  the following set of data:
  \begin{itemize}
  \item Hilbert \(D\)\nb-modules~\(\gamma_x\) for objects~\(x\)
    of~\(\Cat\);
  \item nondegenerate \Star{}homomorphisms \(\varphi_x\colon
    A_x\to\Bound(\gamma_x)\) for objects~\(x\) of~\(\Cat\);

  \item linear maps \(S_g\colon \Hilm_g\to \Bound(\gamma_y,\gamma_x)\)
    for arrows \(g\colon x\to y\) in~\(\Cat\);
  \end{itemize}
  such that
  \begin{enumerate}
  \item for each arrow \(g\colon x\to y\), \(S_g\) is
    \(A_x\)-\(A_y\)-linear, compatible with inner products, and
    nondegenerate:
    \begin{enumerate}
    \item \(S_g(a_1 \xi a_2) =
      \varphi_x(a_1) S_g(\xi) \varphi_y(a_2)\) for \(a_1\in A_x\),
      \(a_2\in A_y\);
    \item \(S_g(\xi_1)^*S_g(\xi_2) =
      \varphi_y(\braket{\xi_1}{\xi_2}_{A_y})\) for all
      \(\xi_1,\xi_2\in \Hilm_g\);
    \item the closed linear span of \(S_g(\Hilm_g)\cdot \gamma_y\)
      is~\(\gamma_x\);
    \end{enumerate}

  \item \(S_{1_x}=\varphi_x\colon A_x\to\Bound(\gamma_x)\) for all
    objects~\(x\);
  \item for each pair of composable arrows \(g\colon y\to z\),
    \(h\colon x\to y\) in~\(\Cat\), \(\xi\in\Hilm_g\),
    \(\eta\in\Hilm_h\), we have \(S_h(\eta)S_g(\xi) =
    S_{gh}(\mu_{g,h}(\eta\otimes\xi))\).
  \end{enumerate}

  Let \((\gamma_x^1,\varphi_x^1,S_g^1)\) and
  \((\gamma_x^2,\varphi_x^2,S_g^2)\) be two such collections.
  Modifications between the corresponding transformations are in
  natural bijection with families of unitaries \(W_x\colon
  \gamma_x^1\to\gamma_x^2\) such that \(W_x\varphi_x^1(a) =
  \varphi_x^2(a)W_x\) for all objects~\(x\) and all \(a\in A_x\) and
  \(W_xS_g^1(\xi) = S_g^2(\xi)W_y\) for all arrows \(g\colon x\to
  y\) in~\(\Cat\) and all \(\xi\in\Hilm_g\).
\end{proposition}

\begin{proof}
  Let \((\gamma_x,V_g)\) as in
  Proposition~\ref{pro:trafo_category-diagram} describe a
  transformation from~\(F\) to~\(\const_D\).  The left
  \(A_x\)\nb-module structure on~\(\gamma_x\) is through a
  nondegenerate \Star{}homomorphism \(\varphi_x\colon
  A_x\to\Bound(\gamma_x)\), and when we record this as extra data, we
  may forget the left module structure on~\(\gamma_x\) and view it
  simply as a Hilbert \(D\)\nb-module.  We also replace the unitary
  \(V_g^*\colon \Hilm_g\otimes_{A_y} \gamma_y \to \gamma_x\) by the
  linear map \(S_g\colon \Hilm_g\to \Bound(\gamma_y,\gamma_x)\)
  defined by \(S_g(\xi)(\eta) \defeq V_g^*(\xi\otimes\eta)\).  The
  map~\(S_g\) satisfies (a)--(c) in~(1) and, conversely, maps~\(S_g\)
  with these three properties are in bijection with isomorphisms of
  correspondences~\(V_g^*\); this is proved in
  \cite{Albandik-Meyer:Product}*{Proposition 2.3}.

  To give a transformation, the unitaries~\(V_g\) for arrows~\(g\)
  in~\(\Cat\) must also satisfy the two conditions in
  Proposition~\ref{pro:trafo_category-diagram}.  The first one
  describes~\(V_{1_x}\), and it gives our condition~(2) when we
  translate it into~\(S_{1_x}\).  The second condition in
  Proposition~\ref{pro:trafo_category-diagram} is the commuting
  diagram~\eqref{eq:cone_category-diagram} that relates \(V_g\)
  and~\(V_h\) to~\(V_{hg}\).  This is equivalent to
  \[
  V_h^*(\eta \otimes V_g^*(\xi\otimes\zeta)) =
  V_{gh}^*(\mu_{g,h}(\eta \otimes \xi)\otimes\zeta)
  \]
  for all \(\xi\in\Hilm_g\), \(\eta\in\Hilm_h\), \(\zeta\in\gamma_z\).
  This is, in turn, equivalent to
  \[
  S_h(\eta)S_g(\xi)(\zeta) = S_{gh}(\mu_{g,h}(\eta\otimes\xi))(\zeta),
  \]
  which is condition~(3).  All these steps may be
  reversed.  So a family \((\gamma_x,\varphi_x,S_g)\) with the
  properties (1)--(3) always comes from a unique transformation.

  The last statement holds
  because~\eqref{eq:cone_modification_category-diagram} commutes for
  given~\((W_x)\) if and only if \(W_xS_g^1(\xi)(\zeta) =
  S_g^2(\xi)W_y(\zeta)\) for all \(\zeta\in\gamma^1_y\).
\end{proof}

The nondegeneracy condition (1).(c) in
Proposition~\ref{pro:cone_as_representation} is the only one with an
unusual form, which we cannot impose as a relation on generators of a
universal \(\Cst\)\nb-algebra.  \emph{If each~\(\Hilm_g\) is proper,}
then this condition is equivalent to a Cuntz--Pimsner covariance
condition for each~\(\Hilm_g\); this is slightly more
general than Theorem~\ref{the:CP_as_colimit} because we are dealing
with a correspondence between two different \(\Cst\)\nb-algebras.  All
proofs carry over to this case, however, and we can now write down a
candidate for the colimit using generators and relations:

\begin{definition}
  \label{def:Cuntz-Pimsner_algebra}
  Let \(\CP(A_x,\Hilm_g,\mu_{g,h})\) be the universal
  \(\Cst\)\nb-algebra generated by the \(\Cst\)\nb-algebra
  \(\bigoplus_x A_x\) and symbols \(S_g(\xi)\) for arrows \(g\colon
  x\to y\) in~\(\Cat\) and \(\xi\in\Hilm_g\), subject to the following
  relations:
  \begin{enumerate}
  \item the relations in the \(\Cst\)\nb-algebra \(\bigoplus_x A_x\)
    hold, \(\Hilm_g\ni\xi\mapsto S_g(\xi)\) is linear for each
    arrow~\(g\), and \(S_{1_x}(a)=a\) for all \(a\in A_x\) and
    all~\(x\);

  \item if \(g\colon x\to y\), \(\xi\in\Hilm_g\), \(a\in A_z\), then
    \[
    S_g(\xi)a=
    \begin{cases}
      S_g(\xi a)&z=y,\\
      0&z\neq y,
    \end{cases}\qquad
    a S_g(\xi)=
    \begin{cases}
      S_g(a\xi)&z=x,\\
      0&z\neq x;
    \end{cases}
    \]
  \item if \(g\colon x\to y\), \(\xi_1,\xi_2\in\Hilm_g\), then
    \(S_g(\xi_1)^*S_g(\xi_2)=\langle \xi_1,\xi_2\rangle_{A_y}\in A_y\);
  \item for \(g\colon x\to y\) and \(a\in A_x\) with
    \(\varphi_{\Hilm_g}(a)\in\Comp(\Hilm_g)\) and for
    \(\xi_j,\eta_j\in\Hilm_g\), the norm of \(a-\sum
    S_g(\xi_j)S_g(\eta_j)^*\) is at most the norm of
    \(\varphi_{\Hilm_g}(a)-\sum \ket{\xi_j}\bra{\eta_j}\) in
    \(\Comp(\Hilm_g)\); here \(\varphi_{\Hilm_g}\colon A_x\to
    \Bound(\Hilm_g)\) denotes the left action;
  \item \(S_h(\eta)S_g(\xi) = S_{gh}(\mu_{g,h}(\eta\otimes\xi))\) for
    all \(\xi\in\Hilm_g\), \(\eta\in\Hilm_h\).
  \end{enumerate}
\end{definition}

It is clear that there is a universal \(\Cst\)\nb-algebra satisfying
these relations.  First, take the universal \Star{}algebra~\(U_1\) on
the set of generators.  Secondly, let~\(U_2\) be the quotient
of~\(U_1\) by the ideal generated by the conditions (1)--(3) and~(5).
Thirdly, take the supremum of all \(\Cst\)\nb-seminorms on~\(U_2\)
that satisfy~(4).  This is the maximal \(\Cst\)\nb-seminorm on~\(U_1\)
that satisfies~(4).  The maximum exists because there is a maximal
\(\Cst\)\nb-seminorm on the \(\Cst\)\nb-subalgebra \(\bigoplus
A_x\subseteq U_2\) and \(\norm{S_g(\xi)}=\norm{\xi}\) for any
\(\Cst\)\nb-seminorm on~\(U_2\) by condition~(3).  Finally,
\(\CP(A_x,\Hilm_g,\mu_{g,h})\) is the (Hausdorff) completion of~\(U_2\)
in this \(\Cst\)\nb-seminorm.

\begin{theorem}
  \label{the:colim_Cuntz-Pimsner}
  Let~\(\Cat\) be a category.  Let \((A_x,\Hilm_g,\mu_{g,h})\) give a
  functor \(F\colon \Cat\to\Corrcat\).  Assume that~\(\Hilm_g\) is a
  \emph{proper} correspondence for each arrow \(g\colon x\to y\).
  Then the \(\Cst\)\nb-algebra \(\CP(A_x,\Hilm_g,\mu_{g,h})\) is a
  colimit of~\(F\) in \(\Corrcat\) and in \(\Corrcat_\prop\).
\end{theorem}

\begin{proof}
  We abbreviate \(\CP\defeq \CP(A_x,\Hilm_g,\mu_{g,h})\).  Condition~(1)
  in Definition~\ref{def:Cuntz-Pimsner_algebra} gives a
  \Star{}homomorphism \(f\colon \bigoplus A_x\to \CP\) and linear maps
  \(S_g\colon \Hilm_g\to\CP\).  Condition~(2) implies \(A_x
  S_g(\Hilm_g)A_y = S_g(\Hilm_g)\) and hence \(A_y S_g(\Hilm_g)^*A_x =
  S_g(\Hilm_g)^*\).  Since all elements in~\(\CP\) may be approximated
  by noncommutative polynomials in elements of \(S_g(\Hilm_g)\),
  \(S_g(\Hilm_g)^*\) for arrows~\(g\) and \(A_x\) for objects~\(x\),
  this implies that the \Star{}homomorphism \(f\colon \bigoplus A_x\to
  \CP\) is nondegenerate.

  Let \(p_x\in\Mult\bigl(\bigoplus A_y\bigr)\) be the projection
  onto~\(A_x\) and let \(\gamma^\CP_x\defeq f(p_x)\CP\); we view this
  right ideal as a Hilbert module over~\(\CP\).  Let~\(A_x\) act
  on~\(\gamma^\CP_x\) on the left via multiplication through~\(f\).
  This is nondegenerate, so~\(\gamma^\CP_x\) becomes a correspondence
  from~\(A_x\) to~\(\CP\).  We may identify \(f(p_x)\cdot\CP\cdot
  f(p_y)\) with \(\Comp(\gamma^\CP_y,\gamma^\CP_x) \subseteq
  \Bound(\gamma^\CP_y,\gamma^\CP_x)\).

  Condition~(2) in Definition~\ref{def:Cuntz-Pimsner_algebra} implies
  \(S_g(\Hilm_g)\subseteq f(p_x)\cdot\CP\cdot f(p_y)\) for \(g\colon
  x\to y\).  Conditions (2) and~(3) say that \(S_g\colon \Hilm_g\to
  \Comp(\gamma^\CP_y,\gamma^\CP_x)\) is a representation of the
  correspondence~\(\Hilm_g\).  They provide an isometric embedding of
  correspondences \(V_g^\CP\colon \Hilm_g\otimes_{A_y} \gamma^\CP_y
  \to \gamma^\CP_x\) by the proof of
  \cite{Albandik-Meyer:Product}*{Proposition 2.3}.

  Our next goal is to show that this isometry is unitary or,
  equivalently, \(S_g(\Hilm_g)\cdot \gamma^\CP_Y\) spans a dense
  subspace of~\(\gamma^\CP_y\).  This argument is essentially the same
  as for one direction in \cite{Albandik-Meyer:Product}*{Proposition
    2.5}.  It is the place where we need the
  correspondences~\(\Hilm_g\) to be proper, that is,
  \(\varphi_{\Hilm_g}(A_x)\subseteq \Comp(\Hilm_g)\).  Let
  \((u_i)_{i\in I}\) be an approximate unit in~\(A_x\).  For each
  \(i\in I\) and \(\epsilon>0\) there is a finite-rank operator \(T =
  \sum_{n=1}^k \ket{\xi_n}\bra{\eta_n}\) on~\(\Hilm_g\) with
  \(\norm{\varphi_{\Hilm_g}(u_i)-T}<\epsilon\).  Condition~(4) ensures
  that
  \[
  \left\lVert \sum_{n=1}^k S_g(\xi_n)S_g(\eta_n)^* - u_i\right\rVert
  <\epsilon.
  \]
  Thus we may approximate~\(u_i x\) by elements of
  \(S_g(\Hilm_g)S_g(\Hilm_g)^* x\subseteq S_g(\Hilm_g)\gamma^\CP_y\)
  for any \(x\in \CP\).  Since the left action of~\(A_x\)
  on~\(\gamma^\CP_x\) is nondegenerate, this shows that
  \(S_g(\Hilm_g)\gamma^\CP_y\) spans a dense subspace
  of~\(\gamma^\CP_x\), as desired.

  We have verified the critical condition (1).(c) in
  Proposition~\ref{pro:cone_as_representation} for the
  correspondences~\(\gamma^\CP_x\) for \(x\in \Cat^0\) and the maps
  \(S_g\colon \Hilm_g\to \Bound(\gamma^\CP_y,\gamma^\CP_x)\).  The
  remaining conditions are built into our relations very directly.  So
  this data comes from a transformation \((\gamma^\CP_x,V_g)\) from
  our diagram~\(F\) to~\(\const_\CP\).

  Now let~\(\Hilm[F]\) be a correspondence from~\(\CP\) to a
  \(\Cst\)\nb-algebra~\(D\).  Then the correspondences
  \(\Hilm[F]_x\defeq \gamma^\CP_x\otimes_\CP \Hilm[F]\) from~\(A_x\)
  to~\(D\) and the isomorphisms of correspondences \(V_g\otimes_\CP
  \Id_{\Hilm[F]} \colon \Hilm_g \otimes_{A_y} \Hilm[F]_y\to
  \Hilm[F]_x\) form a transformation \(F\to\const_D\).  We
  claim that this construction gives an equivalence
  between the groupoid of correspondences \(\CP\to D\) and the
  groupoid of transformations \(F\to\const_D\).

  Let \((\gamma_x,S_g)\) be the data of a transformation
  to~\(\const_D\) for some \(\Cst\)\nb-algebra~\(D\).  Let
  \(\gamma\defeq \bigoplus_x \gamma_x\) with the canonical
  representation of~\(\bigoplus A_x\), as in
  Proposition~\ref{pro:product_coproduct_in_Corr}.  Also map
  \(S_g(\xi)\in\Bound(\gamma_y,\gamma_x)\) to an operator
  on~\(\gamma\) that vanishes on \(\gamma_z\) for \(z\neq y\).  We
  claim that this defines a \Star{}homomorphism \(\alpha\colon
  \CP\to\Bound(\gamma)\), which is nondegenerate because already its
  restriction to~\(\bigoplus A_x\) is nondegenerate.  We want to use
  the universal property of~\(\CP\), of course.  All conditions except
  the fourth one are evident.  To check that one, we copy the other
  half of the proof of \cite{Albandik-Meyer:Product}*{Proposition
    2.5}.

  Let \(g\colon x\to y\) be an arrow, let \(a\in A_x\),
  \(\xi_i,\eta_i\in \Hilm_g\), and let \(C>0\) be strictly bigger than
  the norm of \(\varphi_{\Hilm_g}(a) - \sum \ket{\xi_i}\bra{\eta_i}\).
  It is convenient to use that the map \(\ket{\xi}\bra{\eta}\mapsto
  S_g(\xi)S_g(\eta)^*\) induces a \Star{}homomorphism
  \(\vartheta_g\colon \Comp(\Hilm_g)\to \Bound(\gamma_x)\).  This is
  nondegenerate because Proposition~\ref{pro:cone_as_representation}
  gives \(\Comp(\Hilm_g)\gamma_x = S_g(\Hilm_g) S_g(\Hilm_g)^*
  \gamma_x \supseteq S_g(\Hilm_g) \gamma_y \supseteq \gamma_x\).

  Since \(aS_g(\zeta) = S_g(\varphi_{\Hilm_g}(a)\zeta)\) for all
  \(a\in A_x\), we get \(a \zeta = \vartheta_g(\varphi_{\Hilm_g}(a))
  \zeta\) for all \(a\in A_x\), \(\zeta\in \Comp(\Hilm_g)D = D\).
  Thus the direct action of~\(A_x\) is equal to \(\vartheta_g\circ
  \varphi_{\Hilm_g}(a)\).  This easily implies the norm estimate~(4)
  in Definition~\ref{def:Cuntz-Pimsner_algebra}.  Hence we get the
  desired nondegenerate \Star{}homomorphism \(\CP\to\Bound(\gamma)\),
  so~\(\gamma\) becomes a correspondence from~\(\CP\) to~\(D\).  By
  construction, the transformation \((\gamma_x^\CP\otimes_\CP \gamma,
  V_g^\CP\otimes_\CP \gamma)\) associated to this
  correspondence~\(\gamma\) is the transformation given by the
  original data~\((\gamma_x,S_g)\).

  Let \((\gamma_x^1,S_g^1)\) and \((\gamma_x^2,S_g^2)\) be
  transformations \(F\to\const_D\).  Form the associated
  correspondences \(\gamma^1\) and~\(\gamma^2\) from~\(\CP\) to~\(D\).
  A family of isomorphisms of correspondences \(W_x\colon
  \gamma_x^1\to\gamma_x^2\) gives a unitary operator \(\bigoplus
  W_x\colon \gamma^1\to\gamma^2\) that intertwines the left actions of
  \(\bigoplus A_x\subseteq\CP\).  Conversely, any such operator
  \(\gamma^1\to\gamma^2\) commutes with the projections~\(f(p_x)\) and
  therefore decomposes as \(\bigoplus W_x\) for isomorphisms of
  correspondences \(W_x\colon \gamma^1_x\to\gamma^2_x\).  The
  operators~\(W_x\) form a modification if and only if they also
  intertwine the actions of \(S_g^1(\xi)\) and \(S_g^2(\xi)\) for all
  \(\xi\in\Hilm_g\) and all arrows~\(g\) in~\(\Cat\).  Since these
  elements together with \(\bigoplus A_x\) generate~\(\CP\), this is
  equivalent to intertwining the representations of~\(\CP\).  Thus
  modifications between functors \(F\to\const_D\) are in bijection
  with isomorphisms of the associated correspondences \(\CP\to D\).
  Hence we have an equivalence of groupoids
  \(\Corrcat^{\Cat}(F,\const_D) \cong \Corrcat(\CP,D)\).

  If the correspondences~\(\gamma_x\) are proper, then \(\bigoplus
  A_x\to\Comp(\gamma)\) and hence \(\CP\to\Comp(\gamma)\) because
  \(\bigoplus A_x\to\CP\) is nondegenerate.  Thus we get a proper
  correspondence from~\(\CP\) to~\(D\).  The converse also holds
  because the correspondences~\(\gamma_x^\CP\) are proper.  Hence
  \(\Corrcat_\prop^{\Cat}(F,\const_D) \cong \Corrcat_\prop(\CP,D)\)
  as well, that is, \(\CP\) is also a colimit in the
  subcategory~\(\Corrcat_\prop\).
\end{proof}

Let us return to the notationally easier case where~\(\Cat\) has only
one object, that is, \(\Cat\) is a monoid~\(P\).  By
Proposition~\ref{pro:functor_category-diagram}, a functor
\(P\to\Corrcat\) is the same as an essential product system over the
opposite monoid~\(P^\op\).

\begin{theorem}
  \label{the:colimit_monoid}
  Let~\(P\) be a monoid.
  View a proper, essential product system over~\(P^\op\) as a functor
  \(P\to\Corrcat_\prop\).  The Cuntz--Pimsner algebra of the product
  system is the colimit of this functor \(P\to\Corrcat_\prop\) both in
  \(\Corrcat_\prop\) and in~\(\Corrcat\).
\end{theorem}

\begin{proof}
  The colimit is given by Theorem~\ref{the:colim_Cuntz-Pimsner} and
  Definition~\ref{def:Cuntz-Pimsner_algebra}.  By construction, it
  is also universal for Cuntz--Pimsner covariant representations of
  the product system.
\end{proof}

\subsection{Colimits over bicategories}
\label{sec:colim_over_bicategory}

If~\(\Cat\) is a category, then diagrams \(\Cat\to\Corrcat_\prop\)
have a colimit by Theorem~\ref{the:colim_Cuntz-Pimsner}.  We are
going to extend this to the case where~\(\Cat\) is only a
bicategory.  The bicategory~\(\Corrcat^{\Cat}\) for a general
bicategory~\(\Cat\) is described, among others, in
\cites{Benabou:Bicategories, Leinster:Basic_Bicategories,
  Buss-Meyer-Zhu:Higher_twisted}.  For the target
bicategory~\(\Corrcat\), there are no serious simplifications
compared to the case of an arbitrary target bicategory; we will,
however, often disregard associators in the following arguments
because they are fairly trivial in~\(\Corrcat\).  For simplicitly,
we first assume that~\(\Cat\) is a strict \(2\)\nb-category.  Any
bicategory is equivalent to a strict one
(see~\cite{Leinster:Basic_Bicategories}), so this is no serious
restriction.

If~\(\Cat\) is a strict \(2\)\nb-category, its arrows and objects
form a category~\(\Cat_1\), and a functor \(F\colon
\Cat\to\Corrcat\) contains a functor \(\Cat_1\to\Corrcat\); the
latter is given by \(\Cst\)\nb-algebras~\(A_x\) for objects~\(x\)
of~\(\Cat\), correspondences~\(\Hilm_g\) from~\(A_x\) to~\(A_y\) for
arrows \(g\colon x\to y\) in~\(\Cat\), isomorphisms of
correspondences \(\mu_{g,h}\colon
\Hilm_h\otimes_{A_y}\Hilm_g\to\Hilm_{gh}\) for composable arrows
\(g\colon y\to z\) and \(h\colon x\to y\), subject to the conditions
in Proposition~\ref{pro:functor_category-diagram}.  In addition, a
functor \(F\colon \Cat\to\Corrcat\) also provides isomorphisms of
correspondences \(v_a\colon \Hilm_g\to\Hilm_h\) for \(2\)\nb-arrows
\(a\colon g\Rightarrow h\), which are compatible with horizontal and
vertical composition.  We refer
to~\cite{Buss-Meyer-Zhu:Higher_twisted}*{\S4.1} for the details,
which play no role in the following.

Describe two functors \(F_i\colon \Cat\to\Corrcat\) for \(i=0,1\) by
the data \((A_x^i,\Hilm_g^i,\mu_{g,h}^i,v_a)\) as above.  A
transformation \(\Phi\colon F_0\to F_1\) between them restricts to a
transformation between their restrictions to~\(\Cat_1\) and thus
provides correspondences \(\gamma_x\colon A^0_x\to A^1_x\) and
isomorphisms of correspondences \(V_g\colon \gamma_x \otimes_{A^1_x}
\Hilm_g^1 \to \Hilm_g^0 \otimes_{A^0_y} \gamma_y\) for arrows
\(g\colon x\to y\) in~\(\Cat\), subject to the conditions in
Proposition~\ref{pro:trafo_category-diagram}.  To be a transformation
on the level of~\(\Cat\), we need no extra data, but extra conditions:
the diagrams
\begin{equation}
  \label{eq:transformation_bicategory}
  \begin{tikzpicture}[yscale=1,xscale=4,baseline=(current bounding box.west)]
    \node (add) at (0,1) {\(\gamma_x \otimes_{A^1_x} \Hilm_g^1\)};
    \node (ad) at (1,1) {\(\gamma_x\otimes_{A^1_x} \Hilm_h^1\)};
    \node (ddc) at (0,0) {\(\Hilm_g^0\otimes_{A^0_y} \gamma_y\)};
    \node (dc) at (1,0) {\(\Hilm_h^0\otimes_{A^0_y} \gamma_y\)};

    \draw[dar] (add) -- node {\(\Id_{\gamma_x}\otimes_{A^1_x} v_a^1\)} (ad);
    \draw[dar] (add) -- node[swap] {\(V_g\)} (ddc);
    \draw[dar] (ad) -- node {\(V_h\)} (dc);
    \draw[dar] (ddc) -- node[swap] {\(v_a^0\otimes_{A^0_y} \Id_{\gamma_y}\)} (dc);
  \end{tikzpicture}
\end{equation}
must commute for all \(2\)\nb-arrows \(a\colon g\Rightarrow h\)
in~\(\Cat\), for
parallel arrows \(g,h\colon x\rightrightarrows y\).  This diagram
commutes automatically if~\(a\) is an identity \(2\)\nb-arrow.

A modification between two transformations \(\Phi_1,\Phi_2\colon
F_0\to F_1\) is defined \emph{exactly} as in
Proposition~\ref{pro:modification_category-diagram}; there is no extra
data and no extra condition to be a modification on the level
of~\(\Cat\).

\begin{definition}
  \label{def:CP_over_bicategory}
  Let \((A_x,\Hilm_g,\mu_{g,h},v_a)\) describe a functor from the
  \(2\)\nb-category~\(\Cat\) to~\(\Corrcat\).  The
  \emph{Cuntz--Pimsner algebra} \(\CP(A_x,\Hilm_g,\mu_{g,h},v_a)\) is
  defined as the quotient of \(\CP(A_x,\Hilm_g,\mu_{g,h})\) (see
  Definition~\ref{def:Cuntz-Pimsner_algebra}) by the relations
  \(S_h(v_a(\xi))= S_g(\xi)\) for all \(2\)\nb-arrows \(a\colon
  g\Rightarrow h\) and all \(\xi\in\Hilm_g\).
\end{definition}

\begin{theorem}
  \label{the:colim_Cuntz-Pimsner_2}
  Let~\(\Cat\) be a \textup{(}strict\textup{)} \(2\)\nb-category and
  let \((A_x,\Hilm_g,\mu_{g,h},v_a)\) give a functor \(F\colon
  \Cat\to\Corrcat_\prop\).  The \(\Cst\)\nb-algebra
  \(\CP(A_x,\Hilm_g,\mu_{g,h},v_a)\) is a colimit of~\(F\) both in
  \(\Corrcat\) and in \(\Corrcat_\prop\).
\end{theorem}

\begin{proof}
  Let \(F_1\colon \Cat_1\to\Corrcat_\prop\) denote the restriction of
  a diagram to the arrows and objects in~\(\Cat\).  A transformation
  \(F\to\const_D\) is also a transformation \(F_1\to\const_D\), and
  the modifications are the same in both cases.  Hence the universal
  \(\Cst\)\nb-algebra for transformations \(F\to\const_D\) is a
  quotient of the one for transformations \(F_1\to\const_D\).  The
  extra relations that we need to divide out are exactly the relations
  \(S_h(v_a(\xi))= S_g(\xi)\) for all \(2\)\nb-arrows \(a\colon
  g\Rightarrow h\) and all \(\xi\in\Hilm_g\): this is exactly what is
  needed to make the diagrams~\eqref{eq:transformation_bicategory}
  commute.
\end{proof}

If~\(\Cat\) is only a bicategory, then functors \(\Cat\to\Corrcat\)
look the same as above, except that now the ``category''~\(\Cat_1\)
is only associative and unital up to certain \(2\)\nb-arrows, which
form part of the data.  The definitions of transformations and
modifications, however, do not contain the associators and unit
transformations.  So the proof of
Theorem~\ref{the:colim_Cuntz-Pimsner} extends to non-associative
``categories,'' and Theorem~\ref{the:colim_Cuntz-Pimsner_2} extends
literally to bicategories.

\section{Inductive limits}
\label{sec:inductive}

Let~\(\Cat\) be the partially ordered set~\((\N,\le)\) viewed as a
category, that is, with a unique arrow \(m\to n\) if \(m\le n\) and
no arrow otherwise.  Diagrams indexed by~\(\Cat\) are called
\emph{inductive systems}, and their colimits \emph{inductive
  limits}.
Such a diagram in~\(\Corrcat\) is given by
\(\Cst\)\nb-algebras~\(A_n\), correspondences \(\Hilm_m^n\colon
A_m\to A_n\) for \(m\le n\), and isomorphisms of correspondences
\(\mu_{m,n,k}\colon \Hilm_m^n\otimes_{A_n} \Hilm_n^k \cong \Hilm_m^k\)
for all \(m\le n\le k\), subject to the following conditions.
First, \(\Hilm_n^n\cong A_n\) and~\(\mu_{m,n,k}\) has to be the
canonical isomorphism if \(m=n\) or \(n=k\).  Secondly, the
maps~\(\mu_{m,n,k}\) are ``associative'' (view them as multiplication
maps).

We may simplify this data, up to isomorphism of diagrams: It is
enough to specify \(\Cst\)\nb-algebras \(A_n\) and
correspondences~\(\Hilm_n^{n+1}\) for \(n\in\N\), with no
constraints on the~\(\Hilm_n^{n+1}\).  We may extend this to a
diagram as above by taking
\[
\Hilm_m^n \cong \Hilm_m^{m+1}\otimes_{A_{m+1}} \Hilm_{m+1}^{m+2}
\otimes_{A_{m+2}} \dotsb \otimes_{A_{n-1}} \Hilm_{n-1}^{n}
\]
for \(m\le n\) (the empty tensor product is interpreted as~\(A_n\)
for \(m=n\)) and letting~\(\mu_{m,n,k}\) be the canonical
isomorphisms.  Conversely, any diagram is isomorphic to one of this
form.

Let \((A_n,\Hilm_n^m,\mu_{n,m,k})\) and
\((B_n,\Hilm[F]_n^m,v_{n,m,k})\) be such diagrams.  We may also
simplify transformations between them.  By definition, a
transformation is given by correspondences \(\Hilm[G]_n\colon A_n
\to B_n\) and isomorphisms of correspondences
\[
w_{m,n}\colon \Hilm_m^n\otimes_{A_n} \Hilm[G]_n\cong
\Hilm[G]_m\otimes_{B_m} \Hilm[F]_m^n
\qquad\text{ for all }m\le n,
\]
subject to compatibility conditions with \(\mu_{m,n,k}\)
and~\(v_{m,n,k}\) for all \(m\le n\le k\), and the condition
that~\(w_{n,n}\) be the canonical isomorphism.  It suffices to
specify the isomorphisms~\(w_{n,n+1}\) for \(n\in\N\), without any
conditions.

Finally, a modification between two such transformations,
\((\Hilm[G]_n,w_{n,n+1})\) and \((\Hilm[G]'_n,w'_{n,n+1})\), is
given by isomorphisms of correspondences \(x_n\colon \Hilm[G]_n\to
\Hilm[G]'_n\) such that \(w_{m,n}\circ (\Id_{\Hilm_m^n}\otimes_{A_n}
x_n) = (x_n\otimes_{B_m} \Id_{\Hilm[F]_m^n})\circ w_{m,n}\) for all
\(m\le n\); but these conditions hold for all \(m\le n\) once they
hold for all \(m\in\N\) and \(n=m+1\).

The simplifications above say that the bicategory of functors
\(\Cat\to\Corrcat\) is equivalent to the bicategory of simplified
functors with simplified transformations and modifications.  In
particular, for colimits it makes no difference whether we work
with full or simplified diagrams.

Our general existence theorem shows that any inductive system of
proper correspondences has a colimit in~\(\Corrcat\).  We claim that
for an inductive system of \Star{}homomorphisms in the usual sense,
this colimit is the same as the usual inductive limit in the
category of \(\Cst\)\nb-algebras.  Thus we consider a diagram
\begin{equation}
  \label{eq:inductive_limit_diagram}
  A_0\xrightarrow{\varphi_0}A_1\xrightarrow{\varphi_1}A_2\xrightarrow{\varphi_2}
  \cdots\xrightarrow{\varphi_{n-1}}A_n\xrightarrow{\varphi_{n}}\cdots,
\end{equation}
where the~\(A_n\) are \(\Cst\)\nb-algebras and the~\(\varphi_n\) are
\(^*\)\nb-homomorphisms.  Let~\(A_\infty\) be the inductive limit
\(\Cst\)\nb-algebra of this diagram in the usual sense, and let
\(\varphi^\infty_n\colon A_n \to A_\infty\) be the canonical
\(^*\)\nb-homomorphisms.

\begin{proposition}
  \label{pro:inductive_limit_colimit}
  The \(\Cst\)\nb-algebra~\(A_\infty\) with the
  maps~\(\varphi^\infty_n\) is also a colimit
  of~\eqref{eq:inductive_limit_diagram} in \(\Corrcat_\prop\)
  and~\(\Corrcat\).
\end{proposition}

\begin{proof}
  Let~\(D\) be a \(\Cst\)\nb-algebra and let \(\Hilm[F]_\infty\colon
  A_\infty\to D\) be a correspondence.  For \(n\in \N\), we define a
  correspondence \(\Hilm[F]_n \defeq A_n\otimes_{\varphi_n^\infty}
  \Hilm[F]_\infty\colon A_n\to D\).  These correspondences together with
  the canonical isomorphisms
  \[
  A_n\otimes_{\varphi_n}\Hilm[F]_{n+1}
  \cong A_n\otimes_{\varphi_n} A_{n+1}
  \otimes_{\varphi_{n+1}^\infty}\Hilm[F]_\infty
  \cong A_n \otimes_{\varphi_{n+1}^\infty\circ \varphi_n} \Hilm[F]_\infty
  \cong \Hilm[F]_n
  \]
  give a transformation from~\eqref{eq:inductive_limit_diagram}
  to~\(\const_D\).  An isomorphism of correspondences
  \(\Hilm[F]_\infty\to\Hilm[F]'_\infty\) induces a modification between
  these associated transformations, so we get a functor from the
  groupoid of correspondences \(A_\infty\to D\) to the groupoid of
  transformations in~\(\Corrcat\) from the
  diagram~\eqref{eq:inductive_limit_diagram} to the constant diagram
  on~\(D\).  We claim that this functor is an equivalence of
  groupoids.

  Let the correspondences \(\Hilm[F]_n\colon A_n\to D\) and the
  isomorphisms of correspondences \(\mu_n\colon A_n\otimes_{\varphi_n}
  \Hilm[F]_{n+1} \to \Hilm[F]_n\) form a transformation
  from~\eqref{eq:inductive_limit_diagram} to the constant diagram
  on~\(D\).  We are going to construct a correspondence
  \(\Hilm[F]_\infty\colon A_\infty\to D\).

  If \(a\in \ker \varphi_n\subseteq A_n\), then \(a\otimes_{\varphi_n}
  \xi=0\) for all \(\xi \in \Hilm[F]_{n+1}\) and hence
  \(ab\otimes_{\varphi_n} \xi = a\otimes_{\varphi_n} b\xi = 0\) for
  all \(b\in A_n\), \(\xi \in \Hilm[F]_{n+1}\).  Since
  \(A_n\otimes_{\varphi_n} \Hilm[F]_{n+1} \cong \Hilm[F]_n\), \(\ker
  \varphi_n\) acts trivially on~\(\Hilm[F]_n\).  Similarly, the kernel of
  \(\varphi_n^{n+m}\colon A_n\to A_{n+m}\) acts trivially
  on~\(\Hilm[F]_n\) because \(\Hilm[F]_n \cong A_n \otimes_{\varphi_n}
  A_{n+1} \otimes_{\varphi_{n+1}} \dotsb \otimes_{\varphi_{n+m-1}}
  \Hilm[F]_{n+m}\).  The union of these kernels is dense in the kernel of
  \(\varphi_n^\infty\), which therefore also acts trivially
  on~\(\Hilm[F]_n\).  Thus we may turn~\(\Hilm[F]_n\) into a
  correspondence~\(\Hilm[F]_n'\) from \(A_n'\defeq A_n/\ker
  \varphi_n^\infty\) to~\(D\).  The maps~\(\varphi_n\) become
  embeddings \(A_n'\to A_{n+1}'\to \dotsb \to A_\infty\), and the
  isomorphisms \(\mu_n\colon A_n\otimes_{\varphi_n} \Hilm[F]_{n+1} \to
  \Hilm[F]_n\) induce isomorphisms \(\Hilm[F]_n' \cong
  A_n'\otimes_{\varphi_n} \Hilm[F]_{n+1}'\).  We use these isomorphisms
  and the embeddings \(A_n'\hookrightarrow A_{n+1}'\) to
  view~\(\Hilm[F]_n'\) as a subspace of~\(\Hilm[F]_{n+1}'\) for each~\(n\).

  Let \(\Hilm[F]_\infty\defeq \varinjlim \Hilm[F]_n'\).  Then
  \(\Hilm[F]_\infty\) is a Hilbert \(D\)\nb-module and the
  \(\Cst\)\nb-algebras~\(A_n'\) act on~\(\Hilm[F]_{\infty}\) because
  \(A_n'\cdot \Hilm[F]_\infty = \Hilm[F]'_n \subseteq \Hilm[F]_\infty\).  The
  left action of~\(A_\infty\) is nondegenerate because
  \(A_\infty\cdot\Hilm[F]_\infty\) contains~\(A_n'\cdot\Hilm[F]_\infty =
  \Hilm[F]_n'\) for each \(n\in\N\), and these subspaces are dense
  in~\(\Hilm[F]_\infty\).  Thus~\(\Hilm[F]_\infty\) is a correspondence
  from~\(A_\infty\) to~\(D\).

  This construction is inverse to the one above because \(\Hilm[F]_n
  \cong A_n\otimes_{\varphi_n^\infty}\Hilm[F]_\infty\).
  Hence~\(A_\infty\) has the universal property of the colimit.
\end{proof}

\begin{example}
  \label{exa:expectation_to_corr}
  Let \(B\subseteq A\) be a nondegenerate \(\Cst\)\nb-subalgebra and
  \(E\colon A\to B\) a conditional expectation.  Then
  \(\braket{a_1}{a_2} = E(a_1^* a_2)\) and the obvious right
  multiplication action of~\(B\) turn~\(A\) into a pre-Hilbert
  \(B\)\nb-module.  The action of~\(A\) on itself by left
  multiplication extends to the completion, giving a
  \(\Cst\)\nb-correspondence~\(A_E\) from~\(A\) to~\(B\).  If
  \(C\subseteq B\) and \(F\colon B\to C\) is a conditional
  expectation as well, then \(F\circ E\colon A\to C\) is a
  conditional expectation and the map \(a\otimes b\mapsto E(a)\cdot
  b\) extends to an isomorphism of \(\Cst\)\nb-correspondences \(A_E
  \otimes_B B_F \cong A_{F\circ E}\).  Thus a decreasing chain of
  nondegenerate \(\Cst\)\nb-subalgebras \(\mathcal{R}_n\subseteq A\)
  with conditional expectations \(\mathcal{R}_n \to
  \mathcal{R}_{n+1}\) defines a functor \((\N,\le)\to\Corr\).  This
  situation is studied in
  \cites{Exel-Lopes:Approximately_proper,Exel-Renault:AF_tail}.  To
  apply our theory, we need proper \(\Cst\)\nb-correspondences.
  Equivalently, the conditional expectations are of finite-index
  type as in~\cite{Watatani:Index}.  In the proper case, the above
  diagram has a colimit by Theorem~\ref{the:colim_Cuntz-Pimsner}; in
  fact, this is isomorphic to the \(\Cst\)\nb-algebra constructed
  in~\cite{Exel-Lopes:Approximately_proper}.  It is an appropriate
  analogue of the inductive limit of a chain of \Star{}homomorphisms
  by Proposition~\ref{pro:inductive_limit_colimit}.
\end{example}

\end{document}